\documentclass[10pt,  reqno]{amsart}
\title[Quasi-bialgebras from set-theoretic type solutions of the YBE]{Quasi-bialgebras from set-theoretic type solutions of the Yang-Baxter equation}
\author[A. Doikou, A. Ghionis and B. Vlaar]{Anastasia Doikou, Alexandros Ghionis and Bart Vlaar}

\address[A. Doikou and A. Ghionis] {Department of Mathematics, Heriot-Watt University,
Edinburgh EH14 4AS $\&$ Maxwell Institute for Mathematical Sciences, Edinburgh, UK}
\email{a.doikou@hw.ac.uk, ag181@hw.ac.uk}

\address[B. Vlaar] {Max Planck Institute for Mathematics
Vivatsgasse 7, 53111 Bonn, Germany}
\email{b.vlaar@hw.ac.uk}

\usepackage[english]{babel}
\usepackage[utf8]{inputenc}
\usepackage{amsmath, amssymb, bm}
\usepackage{graphicx}
\usepackage{subcaption}
\usepackage[colorinlistoftodos]{todonotes}
\usepackage{indentfirst}

\newcommand{\hiddenpower}[2] { \ifnum \numexpr#2=1 #1 \else #1^#2 \fi }
\numberwithin{equation}{section}

\def\be{\begin{equation}}
\def\ee{\end{equation}}
\def\ba{\begin{eqnarray}}
\def\ea{\end{eqnarray}}

\def\non{\nonumber}

\newcommand{\cal}{\mathcal}

\usepackage{xparse}
\ExplSyntaxOn
\newcounter{diff_order}
\newcounter{diff_power}

\newcommand{\rawdiff}[3]
{
	\setcounter{diff_order}{0}
	\clist_map_inline:nn{#3}{\stepcounter{diff_order}}
	
	\frac{\hiddenpower{#1}{\thediff_order} #2}
	{
		\def\old_var{DefaultValue}
		\setcounter{diff_power}{0}
		
		\clist_map_inline:nn{#3}
		{
			\def\new_var{##1}
			\ifnum \thediff_power=0
				\stepcounter{diff_power}
			\else
				\tl_if_eq:NNTF \new_var \old_var
				{\stepcounter{diff_power}}
				{
					#1 \hiddenpower{\old_var}{\thediff_power}
					\setcounter{diff_power}{1}
				}
			\fi

			\def\old_var{##1}
		}
		
		#1 \hiddenpower{\old_var}{\thediff_power}
	}
}
\ExplSyntaxOff

\newlength{\bibitemsep}
\newlength{\bibparskip}
\let\oldthebibliography\thebibliography
\renewcommand\thebibliography[1]{%
  \oldthebibliography{#1}%
  \setlength{\parskip}{\bibitemsep}%
  \setlength{\itemsep}{\bibparskip}%
}

\newtheorem{thm}{Theorem}[section]
\newtheorem{lemma}[thm]{Lemma}
\newtheorem{cor}[thm]{Corollary}
\newtheorem{pro}[thm]{Proposition}
\newtheorem{defn}[thm]{Definition}

\newtheorem{rem}[thm]{Remark}
\newtheorem{exa}[thm]{Example}
\newtheorem{conj}[thm]{Conjecture}

\begin{document}
\vskip 0.3in
\hfill
\begin{abstract}
We examine classes of quantum algebras emerging from involutive, non-degenerate 
set-theoretic solutions of the Yang-Baxter equation and their $q$-analogues. 
After providing some universal results on quasi-bialgebras and 
admissible Drinfeld twists we show that the quantum algebras produced from set-theoretic solutions and their 
$q$-analogues are in fact 
quasi-triangular quasi-bialgebras. Specific illustrative examples compatible with our generic findings are worked out.  
In the $q$-deformed case of set-theoretic solutions we also construct  admissible Drinfeld twists similar to the set-theoretic ones, 
subject to certain extra constraints dictated by the $q$-deformation. These findings greatly generalize recent  relevant results on set-theoretic solutions and their $q$-deformed analogues.\\
\\
{\it 2010 MCS: 16T20, 16T25, 17B37, 20G42}
\end{abstract}

\maketitle

\date{}
\vskip 0.2in

\section*{Introduction}\footnote{{\it \underline{Data availability statement}: Data sharing not applicable to this article as no datasets were generated or analysed during the current study.}}
\noindent 
The primary focus in this article is the investigation of classes of quantum algebras arising 
from set-theoretic solutions of the Yang-Baxter equation (YBE)  \cite{Baxter, Yang} 
and their $q$-deformed analogues,  and their connections to quasi-triangular quasi-bialgebras.
The problem of identifying  and classifying set-theoretic solutions of the Yang-Baxter equation 
was first  posed  by Drinfeld in early 90s
\cite{Drin} and ever since a considerable research activity has been developed on this topic 
(see for instance \cite{ESS, Eti, Hienta}, \cite{ABS, Papag, Veselov}). 
A lot of attention has been focused recently  on certain algebraic structures that generalize 
nilpotent  rings, called braces and
generate all involutive,  non-degenerate solutions of the YBE \cite{[25], [26]}.
Skew-braces on the other hand are used \cite{GV} 
to describe non-involutive set-theoretic solutions of the YBE,  and they may be instrumental 
in identifying universal $R$-matrices. This rising research field  has been particularly 
fertile and numerous relevant studies have 
been produced  over the past several years (see for instance
 \cite{bcjo, bach2, Catino, [6]},  \cite{DoiSmo1}-\cite{Doikoutw}, \cite{Gateva}-\cite{15}, 
\cite{Gobadi1, Gobadi2}, \cite{GV, JKA, Jesp2, LebVend}, \cite{Smo1}-\cite{LAA}).

Novel connections between set-theoretic solutions, quantum
integrable systems and the associated quantum algebras were uncovered in \cite{DoiSmo1, DoiSmo2}. 
More precisely, quantum groups associated to Baxterized solutions of the Yang-Baxter equation coming 
from braces were derived via the FRT construction \cite{FadTakRes},
new classes of quantum discrete integrable systems with periodic and open boundary conditions were produced and  the symmetries of these integrable systems were also identified.
Furthermore, the explicit forms of admissible twists for involutive set-theoretic solutions of the YBE were derived
and their admissibility was proven in \cite{Doikoutw}. Admissible twists for non-involutive 
set-theoretic solutions coming from skew braces were also subsequently introduced  in \cite{Gobadi2}.

The notion of an admissible twist ${\cal F}$, which links distinct  Hopf or quasi-Hopf algebras 
was originally introduced by Drinfeld in \cite{Drintw}. 
If the Hopf algebra
is in addition equipped with a quasi-triangular structure,  i.e.  a universal $R$-matrix exists, 
then the twisted Hopf algebra also has a quasi-triangular
structure.
In  Drinfeld's works such twists as well as the associated universal $R$-matrices are always considered 
to have semi-classical limits \cite{Drinfeld, Drintw},  i.e. 
they can be expressed as formal series expansions in powers of some deformation parameter, 
with the leading term being the unit element.
In the analysis of \cite{DoiSmo1, DoiSmo2, Doikoutw} on the other hand  Baxterized $R$-matrices coming from set-theoretic solutions 
of the Yang-Baxter equation were identified,  being of the form
$R(\lambda) = r+ {1\over \lambda} {\cal P}$, where $r$ is the set-theoretic solution of the
Yang-Baxter equation and ${\cal P}$ is the permutation operator. 
Interestingly, the $r$-matrix in this case does not contain any free parameter (deformation parameter), 
and consequently the $R$-matrix has no semi-classical analogue. 
A similar observation can be made about the associated admissible twist \cite{Doikoutw}.
This is a crucial difference between the studies in \cite{DoiSmo1, DoiSmo2, Doikoutw} and Drinfeld's analysis  \cite{Drintw}.
Moreover, whenever the notion of the  twist is discussed in Drinfeld's original work and in the literature in general
a quite restrictive action of the co-unit on the twist is almost always assumed, i.e.
 $(\mbox{id} \otimes \epsilon) {\cal F} = (\epsilon \otimes \mbox{id} ) {\cal F}=1_{\cal A}$ (${\cal A}$ is the associated quantum algebra).  It is however worth noting that in \cite{Drintw} Drinfeld describes how to use certain twists without this restricted counit action to twist quasi-bialgebras with nontrivial unit constraints to quasi-bialgebras to trivial constraints.  For our purposes here we allow quasi-bialgebras with nontrivial unit constraints.  We will further discuss this point in subsection 1.2.

It was shown in \cite{Doikoutw} that all involutive, set-theoretic solutions of the YBE 
can be obtained from the permutation operator via a suitable admissible twist that was explicitly derived.
The action of the co-unit on the twist  in this case was also identified and is given as 
$(\mbox{id} \otimes \epsilon) {\cal F} =1_{\cal A}$ and $(\epsilon \otimes \mbox{id} ) {\cal F}\neq 1_{\cal A}.$
Bearing these findings in mind, we relax in this study the restrictive condition
$(\mbox{id} \otimes \epsilon) {\cal F} = (\epsilon \otimes \mbox{id} ) {\cal F}=1_{\cal A}$
and consider a general scenario, which leads to  intriguing new results. 
One would expect that the quantum algebra emerging from any $R$-matrix that satisfies the Yang-Baxter equation
would  be a bialgebra. It was noted however in \cite{Doikoutw}, using a special class of set-theoretic solutions,
 that the corresponding quantum algebra was not co-associative and the related associator 
was derived. This was indeed the first hint that set-theoretic solutions of the YBE give rise to quasi-bialgebras.
In this article we greatly extend these findings after providing some universal results on 
quasi-bialgebras and admissible Drinfeld twists. 

More specifically, we describe below what is achieved in each of the subsequent sections.

\begin{enumerate}

\item In section 1 we show some new results on quasi-bialgebras and general admissible Drinfelds's 
twists that lay out the main frame
for studying quantum algebras arising from the set-theoretic and $q$-deformed set-theoretic solutions of the YBE.
More precisely, in subsection 1.1  before we present the main results we first recall fundamental definitions 
on quasi-bialgebras and quasi-triangular quasi-bialgebras.  We
then  show that given certain imposed conditions the non-associative version 
of the YBE reduces to the familiar YBE; however the underlying  
quantum algebra is still a quasi-bialgebra.

In subsection 1.2 we move on to study 
generic admissible twists by relaxing the conditions 
$(\mbox{id} \otimes \epsilon) {\cal F} = (\epsilon \otimes \mbox{id} ) {\cal F}=1_{\cal A}$ 
and we also examine the consequences of such a general choice.
We then give particular emphasis on cases where the twists send a  quasi-bialgebra to a bialgebra. 
These cases are relevant to the findings of sections 2 and 3. An explicit illustrative example is also 
worked out at the end of the subsection.

\item  In section 2 we focus on the quantum algebras emerging from involutive, non-degenerate 
set-theoretic solutions of the YBE. 
More specifically,  in subsection 2.1 we present some background information regarding 
set-theoretic solutions of the 
Yang-Baxter equation as well as a brief review on the recent findings of \cite{DoiSmo1}
on the links between set-theoretic solutions of the Yang-Baxter equation and quantum algebras.

In subsection 2.2 we review basic definitions and facts about the $\mathfrak{gl}_n$ Yangian, useful for the purposes of our analysis.
We first recall the definition of the $\mathfrak{gl}_n$ Yangian, 
which is most relevant in our present investigation \cite{Drinfeld, ChaPre}.  We then recall that the Yangian is a Hopf algebra and we comment on the action of the antipode after suitably twisting the algebra.

In subsection 2.3 we study 
set-theoretic solutions of the Yang-Baxter equation associated to quasi-bialgebras.
We first review some fundamental results on the admissible Drinfeld twist for involutive,
 set-theoretic solution of the YBE derived in \cite{Doikoutw}.
Specifically,  it  was shown in \cite{Doikoutw} that Baxterized set-theoretic  solutions are always 
coming from the $\mathfrak{gl}_n$
Yangian $R$-matrix via a suitable twist.  We show here a key proposition by introducing a family of group-like elements commuting with the Yangian $R$-matrix, which leads to the main statement that the respective 
twisting of the $\mathfrak{gl}_n$ Yangian yields a
quasi-triangular quasi-bialgebra in accordance to the findings of section 1.  
A special case of set-theoretic solution called the Lyubashenko solution \cite{Drin} is also presented as an illustrative example.

\item In section 3 we discuss the $q$-generalizations of set-theoretic solutions.  Although these generalizations  strictly speaking are not set-theoretic solutions of the Yang-Baxter equation 
they are certainly inspired by the results of \cite{DoiSmo2} and \cite{Doikoutw}.  First, in subsection 3.1 we provide basic definitions and information regarding the 
algebra ${\mathfrak U}_q(\widehat{\mathfrak{gl}_n})$ \cite{Drinfeld, Jimbo, Jimbo2} and we also comment on the action of the antipode after suitably twisting the algebra.

In subsection 3.2,  motivated by the set-theoretic solutions and the associated twists \cite{Doikoutw},  we generalize 
results regrading the twist of the ${\mathfrak U}(\widehat{\mathfrak{gl}_n})$ $R$-matrix.
We  exploit the admissible twists found for the set-theoretic solutions to derive generalized 
$q$-deformed solutions. The findings of this section greatly generalize the preliminary results of \cite{DoiSmo2} and produce a generic class of solutions, i.e. the $q$-deformed alanogues of the set-theoretic solutions. The $q$-analogue of Lyuabshenko's solution, first introduced in \cite{DoiSmo2}, 
is also discussed as an example of this construction.

\end{enumerate}

\section{Quasi-triangular quasi-bialgebras}

\noindent
In this section we recall fundamental definitions on quasi-bialgebras and quasi-triangular quasi-bialgebras 
(see also for instance \cite{Categorical, ChaPre, Drintw,  Kassel, Majid}),
 and we show various novel propositions on special classes of quasi-bialgebras as well as  on generic admissible Drinfeld twists.
Particular emphasis is given on Drinfeld twists that send a quasi-bialgebra to a bialgebra as this case is pertinent to the findings of sections 2 and 3.

\subsection{Quasi-bialgebras and the usual YBE}

\noindent Before we present the main results we first recall
some fundamental definitions on quasi-bialgebras and the associated notion of quasi-triangularity.
Throughout Section 1, $k$ is any field.  Later on we will restrict to the case $k=\mathbb{C}$.
It is worth noting that if ${\cal A}$ is a $ k$-algebra then ${\cal A}\otimes {\cal A}$ (in fact ${\cal A}^{\otimes n},\ n \in \{1, 2, 3, \ldots \}$)
is a $k$-algebra in a natural way. If $V$ and $W$ are ${\cal A}$-modules, then $V\otimes 
W$ is an ${\cal A} \otimes {\cal A}$-module and
not necessarily an ${\cal A}$-module. The following definition due to Drinfeld \cite{Drintw} provides a
general setup where the category $\mbox{Rep}({\cal A})$ of all ${\cal A}$-modules naturally has a tensor structure.
\begin{defn}{\label{definition1}} A quasi-bialgebra $\big ({\cal  A}, \Delta, \epsilon, \Phi, c_r, c_l \big )$ is a 
unital associative algebra ${\cal A}$ over some field $k$ with the following algebra homomorphisms:
\begin{itemize}
\item the co-product $\Delta: {\cal A} \to {\cal A} \otimes {\cal A}$
\item the co-unit $\epsilon: {\cal A} \to k$
\end{itemize}
together with the invertible element $\Phi\in {\cal A} \otimes {\cal A} \otimes {\cal A}$ 
(the associator) and the invertible elements $c_l, c_r \in {\cal A}$ (unit constraints),  such that:
\begin{enumerate}
\item $(\mbox{id} \otimes \Delta) \Delta(a)  = \Phi \Big ((\Delta \otimes \mbox{id}) \Delta(a)\Big ) \Phi^{-1},$ $\forall a \in {\cal A}.$
\item $\Big ((\mbox{id} \otimes \mbox{id} \otimes \Delta)\Phi \Big ) \Big ((\Delta \otimes \mbox{id} \otimes \mbox{id})\Phi \Big )= 
\Big (1\otimes \Phi \Big )  \Big ((\mbox{id} \otimes \Delta \otimes \mbox{id} )\Phi \Big ) \Big  (\Phi \otimes 1 \Big ).$
\item $(\epsilon \otimes \mbox{id})\Delta(a) = c_l^{-1} a c_l$ and $(\mbox{id} \otimes \epsilon)\Delta(a) = c_r^{-1} a c_r,$ 
$\forall a  \in {\cal A}.$
\item $(\mbox{id} \otimes \epsilon \otimes \mbox{id})\Phi = c_r \otimes c_l^{-1}.$
\end{enumerate}
\end{defn}
In the special case where $\Phi = 1 \otimes 1 \otimes 1$ one recovers a bialgebra, i.e.  co-associativity is restored.

Before we move on to the definition of a quasi-triangular quasi-bialgebra we first introduce some useful notation. 
Let $\sigma: {\cal A} \otimes {\cal A} \to {\cal A} \otimes {\cal A}$ be the 
``flip'' map, such that $a\otimes b \mapsto b \otimes a$ $\forall a,b \in {\cal A},$ 
then we set $\Delta^{(op)} := \sigma \circ \Delta.$ A quasi-bialgebra is called {\it cocommutative} if $\Delta^{(op)}= \Delta.$ 
We also  consider the general element $A = \sum_j a_j \otimes b_j \in {\cal A} \otimes {\cal A}$, then in the ``index'' notation we denote:
$A_{12} :=  \sum_j a_j\otimes b_j  \otimes 1,$ $A_{23} :=\sum_j 1\otimes   a_j\otimes b_j $ and $A_{13} :=  \sum_j a_j \otimes 1\otimes b_j .$ In fact, given an algebra ${\cal A}$ with homomorphisms $\Delta: {\cal A} \to {\cal A}\otimes {\cal A}$ and $\epsilon: {\cal A} \to k$, the conditions (1)-(4) listed in the Definition 1.1 are equivalent to the statement that the category of ${\cal A}$-modules,  equipped with the tensor product of the category of $k$-linear spaces,  is itself a tensor category (see e.g.  \cite[Proposition XV.1.2]{Kassel}).

By means of the axioms of Definition \ref{definition1} we can derive further counit formulas 
for the associator and unit constraints.
\begin{lemma}{\label{lemma0}}
Let  $\big ({\cal  A}, \Delta, \epsilon, \Phi, c_r, c_l \big )$ be a quasi-bialgebra, then:
\begin{eqnarray}
&& (\epsilon \otimes \mbox{id} \otimes \mbox{id})\Phi = \Delta(c_l^{-1})(c_l\otimes 1), ~~~~~~
(\mbox{id} \otimes \mbox{id} \otimes \epsilon)\Phi =(1 \otimes c_r^{-1})\Delta(c_r) \label{9}\\
&& \epsilon(c_l) = \epsilon(c_r) \label{10}
\end{eqnarray}
\end{lemma} 
\begin{proof}
By applying the counit on the second and third tensor factors in axioms (2) and using axiom (3), we obtain
(\ref{9}). Hence, applying the counit on the remaining (or third) tensor factor in axiom (4) and using axiom (3) we obtain
(\ref{10}).
\end{proof}

\begin{rem}{\label{remco} }We may define the normalized quantities $\hat c_l = \epsilon(c_l^{-1}) c_l$ and $\hat c_r = \epsilon(c_r^{-1}) c_r.$ Then  
$\big ({\cal  A}, \Delta, \epsilon, \Phi, \hat c_r, \hat c_l \big )$ is a quasi-bialgebra over $k$ with $\epsilon(\hat c_l) = \epsilon(\hat c_r) = 1;$ 
indeed the normalization of $c_l,\ c_r$ leaves the quasi-bialgebra axioms invariant.
Considering the latter statement we may assume, without loss of generality, that
$\epsilon(c_r) = \epsilon(c_l)=1$. In this case axiom (4) of Definition (\ref{definition1}) implies that we may recover $c_l$ and $c_r$ from $\Phi$ via
\begin{equation}
c_l^{-1} = \big ( \epsilon \otimes \epsilon \otimes \mbox{id}\big)\Phi, ~~~~c_r = \big ( \mbox{id} \otimes \epsilon \otimes \epsilon \big)\Phi. \label{constr0}
\end{equation}
Henceforth, we will always assume that $\epsilon(c_r) = \epsilon(c_l)=1$ and define a quasi-bialgebra structure
just by specifying $({\cal A}, \Delta, \epsilon, \Phi);$ in particular, we impose the axioms of Definition \ref{definition1}, where $c_l$ and $c_r$ are
shorthand notations defined in (\ref{constr0}). It is worth emphasizing that the common assumption for quasi-bialgebras is  $c_l=c_r=1$ \cite{Drintw}, so
 even though we have a counit-normalization here we still discuss  a more general situation (see also  later in the text the relevant Corollary \ref{trivial}).
\end{rem}

Following \cite{Drintw}, see also \cite{Kassel}, the notion of quasitriangularity for bialgebras extends to quasi-bialgebras and can be defined category-theoretically. The following definition captures the equivalent algebraic characterization.
\begin{defn}{\label{definition2}}  A quasi-bialgebra $\big ({\cal  A}, \Delta, \epsilon, \Phi \big )$ is called 
quasi-triangular (or braided) if an invertible element ${\cal R} \in {\cal A} \otimes {\cal A}$ (universal $R$-matrix) exists, such that
\begin{enumerate}
\item $\Delta^{(op)}(a) {\cal R} = {\cal R} \Delta(a),$ $\forall a \in {\cal A}.$
\item $(\mbox{id} \otimes \Delta){\cal R} = \Phi^{-1}_{231} {\cal R}_{13} \Phi_{213} {\cal R}_{12} \Phi^{-1}_{123}.$
\item $(\Delta \otimes \mbox{id}){\cal R} = \Phi_{312} {\cal R}_{13} \Phi_{132}^{-1} {\cal R}_{23} \Phi_{123}.$
\end{enumerate}
\end{defn}

By employing conditions (1)-(3) of Definition \ref{definition2} and condition (3) of Definition \ref{definition1} 
one deduces that $(\epsilon \otimes \mbox{id}) {\cal R}= c_r^{-1} c_l$ and $(\mbox{id} \otimes \epsilon) {\cal R}= c_l^{-1} c_r$.
Moreover, by means of conditions (1)-(3) of Definition \ref{definition2}, it follows that ${\cal R}$ satisfies a non-associative
version of the Yang-Baxter equation
\begin{equation}
{\cal R}_{12} \Phi_{312} {\cal R
}_{13} \Phi_{132}^{-1} {\cal R}_{23} \Phi_{123} =  
\Phi_{321}{\cal R}_{23}\Phi^{-1}_{231} {\cal R}_{13} \Phi_{213} {\cal R}_{12}. \label{modYBE}
\end{equation}

If ${\cal A}$ is a quasi-bialgebra in Definition \ref{definition2} then $\big ({\cal  A}, \Delta, \epsilon, \Phi, {\cal R} \big )$ 
is a quasi-triangular quasi-bialgebra.
The quasi-bialgebra is called  {\it triangular} if in addition to the conditions (1)-(3) of Definition \ref{definition2}, 
${\cal R}$ also satisfies ${\cal R}_{21} {\cal R}_{12}=1$.

In the special case where $\Phi = 1\otimes 1 \otimes 1$ one recovers a quasi-triangular bialgebra.
Indeed, the co-associativity is restored  $(\mbox{id} \otimes \Delta) \Delta(a)  =
(\Delta \otimes \mbox{id}) \Delta(a),$ and $c_r =c_l =1.$
Also,  $(\mbox{id} \otimes \Delta){\cal R} = {\cal R}_{13} {\cal R}_{12},$ $(\Delta \otimes \mbox{id}){\cal R} 
= {\cal R}_{13} {\cal R}_{23},$ and the universal ${\cal R}$-matrix satisfies the usual Yang-Baxter equation 
\begin{equation}
{\cal R}_{12}  {\cal R}_{13}  {\cal R}_{23}  = {\cal R}_{23} {\cal R}_{13} {\cal R}_{12}.  \label{YBE}
\end{equation}

In the remainder of this paper we consider two special cases of the general algebraic setting for quasi-triangular quasi-bialgebras as described above.
This specific setup, introduced in the following proposition, will be central in proving general key properties in this section and will be particularly relevant to the findings of sections 2 and 3.

\begin{pro}{\label{lemma1}} Let $\big ({\cal  A}, \Delta, \epsilon, \Phi, {\cal R} \big )$  
be a quasi-triangular quasi-bialgebra, then the following two statements hold:
\begin{enumerate}
\item Suppose $\Phi$  satisfies the condition (in the index notation)
$\Phi_{213}{\cal R}_{12} = {\cal R}_{12} \Phi_{123},$ then
\begin{eqnarray}
(\mbox{id} \otimes \Delta){\cal R} =  \Phi^{-1}_{231} {\cal R}_{13}{\cal R}_{12}, \quad   
(\Delta \otimes \mbox{id}){\cal R} = {\cal R}_{13}  {\cal R}_{23} \Phi_{123}, \label{2}
\end{eqnarray}
and the universal  ${\cal R}$ matrix satisfies the usual Yang-Baxter equation.\\ Also,
$(\epsilon \otimes \mbox{id}){\mathcal R} =  c_l,$ $~(\mbox{id} \otimes \epsilon){\cal R} = c_l^{-1} $ and $c_r =1_{\cal A}.$

\item  Suppose $\Phi$  satisfies  the condition $ \Phi_{132} {\cal R}_{23} ={\cal R}_{23}\Phi_{123},$
then
\begin{eqnarray}
 (\mbox{id} \otimes \Delta){\cal R} =  {\cal R}_{13}{\cal R}_{12}\Phi^{-1}_{123}, \quad
 (\Delta \otimes \mbox{id}){\cal R} = \Phi_{312} {\cal R}_{13}  {\cal R}_{23},  \label{22}
\end{eqnarray}
and the universal  ${\cal R}$ matrix satisfies the usual Yang-Baxter equation.\\ Also,
$(\epsilon \otimes \mbox{id}){\mathcal R} = c_r^{-1},$ $~(\mbox{id} \otimes \epsilon){\cal R} = 
c_r$ and $c_l=1_{\cal A}.$ 
\end{enumerate}
\end{pro}
\begin{proof} The proof of (\ref{2})  and   (\ref{22}) is straightforward by means of relations (2) and (3) 
of Definition \ref{definition2} and the conditions of cases (1), (2) of the Proposition.

By considering case (1) of Proposition \ref{lemma1} we conclude that the modified Yang-Baxter equation (\ref{modYBE}) becomes 
\begin{equation}
{\cal R}_{12}  {\cal R}_{13}  {\cal R}_{23} \Phi_{123} = 
{\cal R}_{23} {\cal R}_{13} {\cal R}_{12}  \Phi_{123}, \label{mod2}
\end{equation}
and because $\Phi$ is an invertible element, equation (\ref{mod2}) reduces to the usual Yang-Baxter equation (\ref{YBE}).
We end up to the same conclusion by considering case (2).\\
The remaining statements are shown as follows:
\begin{enumerate}
\item  We act on the second of the equations (\ref{2}) with $\mbox{id} \otimes \epsilon \otimes \mbox{id}$:
\begin{eqnarray}
&& (\mbox{id} \otimes \epsilon \otimes \mbox{id})\big ( ( \Delta \otimes \mbox{id} ){\cal R} \big)= (\mbox{id} \otimes \epsilon
\otimes \mbox{id}) \big ({\cal R}_{13} {\cal R}_{23} \Phi_{123} \big )\ \Rightarrow \nonumber \\
&&(c_r^{-1}  \otimes 1) {\cal R}_{13} (c_r  \otimes 1) =  {\cal R}_{13}\big ( (\mbox{id} \otimes 
\epsilon \otimes \mbox{id}) {\cal R}_{23}\big ) (c_r \otimes c_l^{-1}) \Rightarrow \nonumber\\
&&  {\cal R}_{13}^{-1} (c_r^{-1}  \otimes 1) 
{\cal R}_{13} (1 \otimes c_l) = (\mbox{id} \otimes \epsilon \otimes \mbox{id}) {\cal R}_{23}. \label{eps1}
\end{eqnarray}
We now act with $\epsilon$ on the first position of the tensor product (\ref{eps1}) and we obtain (recall $\epsilon(c_l) = \epsilon(c_r) = 1$)
$(\epsilon \otimes \mbox{id}) {\cal R} = c_l.$ By substituting the latter on (\ref{eps1}) we conclude that $c_r =1_{\cal A}.$

Similarly, we can act on the first of equs. $(\ref{2})$  
with $ \mbox{id}\otimes  \mbox{id} \otimes \epsilon,$ and arrive at 
\begin{equation}
( \mbox{id}\otimes  \mbox{id} \otimes \epsilon){\cal R}_{13} = (c_l^{-1} \otimes 1) {\cal R}_{12} 
(1 \otimes c_r) {\cal R}_{12}^{-1}, \nonumber
\end{equation}
which after acting with $\epsilon$ on the second position of the tensor product leads to  $(\mbox{id} \otimes \epsilon){\cal R}= c_l^{-1}$ and $c_r=1_{\cal A}$.

\item Likewise,  assuming that $\Phi_{132} {\cal R}_{23} ={\cal R}_{23}\Phi_{123}$ holds we show, following the logic of the proof above, that
$(\epsilon \otimes  \mbox{id}){\cal R} = c_r^{-1},$ $~( \mbox{id}\otimes \epsilon){\cal R} = c_r$ and $c_l= 1_{\cal A}.$
\hfill \qedhere 
\end{enumerate}

\end{proof}

\begin{rem}{\label{specialcase}} Let $u \in {\cal A}$ be a group-like element, i.e. $u$ is invertible and
$\Delta(u) = u \otimes u,$ and consequently $\epsilon(u) =1.$ Let also 
$\Phi = 1 \otimes 1 \otimes u^{-1} $ and ${\cal R} = u^{-1} \otimes u.$  With this choice of $\Phi$ and ${\cal R}$ a 
quasi-triangular quasi-bialgebra structure is defined.  
Indeed, the case (1) of Proposition \ref{lemma1} holds, i.e. conditions (\ref{2}) are satisfied and ${\cal R}$ trivially satisfies the 
Yang-Baxter equation. Also,  $c_r =1_{{\cal A}},$ $c_l = u$ and $(\epsilon \otimes \mbox{id}) {\cal R} = u,$ 
$(\mbox{id} \otimes \epsilon){\cal R} = u^{-1}.$  
\end{rem}

\subsection{Drinfeld twists} 

One of the main results on quasi-triangular (quasi-)bialgebras as shown 
by Drinfeld \cite{Drinfeld, Drintw} is the fact that  
the property of being quasi-triangular  (quasi-)bialgebra is
preserved by twisting (see also \cite{MaiSan}). As far as we can tell 
whenever the notion of Drinfeld twist is discussed in the literature
a trivial action of the co-unit on the twist  is almost always assumed.
In the following proposition
we are going to relax this condition and consider the most general scenario. 
In \cite{Drintw} (Drinfeld 1989, Section 1) it is explained how to use certain Drinfeld 
twists without this restricted counit action to twist quasi-bialgebras with nontrivial unit constraints to quasi-bialgebras to trivial constraints. For our purposes it is convenient to allow quasi-bialgebras with nontrivial unit constraints. We come back to this point later. We are then going to discuss a specific 
case that is associated to the study 
of quantum algebras emerging from set-theoretic solutions of the 
Yang-Baxter equation corresponding also to the conditions of Proposition \ref{lemma1}.

The following result is the natural extension of Drinfeld's result \cite{Drintw} on twists for quasi-bialgebras to the case of nontrivial unit constraints. Recall that we assume that, without loss of generality,  the unit constraints are expressed in terms of the associator by means of equation (\ref{constr0}).

\begin{pro}{\label{twist}} Let $\big ({\cal A},\Delta, \epsilon, \Phi, {\cal R} \big )$ be a 
quasi-triangular quasi-bialgebra and let ${\cal F} \in {\cal A} \otimes {\cal A}$ be an invertible element, such that
\begin{eqnarray}
&& \Delta_{\cal F}(a) = {\cal F} \Delta(a) {\cal F}^{-1}, ~~\forall a \in {\cal A} \\
&& \Phi_{\cal F}({\cal F} \otimes 1) \big  ( (\Delta \otimes \mbox{id}){\cal F} \big ) =  
 (1\otimes {\cal F})\big  ( (\mbox{id} \otimes \Delta){\cal F}\big ) \Phi \label{fi}\\
&& {\cal R}_{\cal F} = {\cal F}^{(op)} {\cal R} {\cal F}^{-1} ,\label{R}
\end{eqnarray}
where ${\cal F}^{(op)} := \sigma({\cal F}).$ 
Then  $\big ({\cal A},\Delta_{\cal F}, \epsilon, \Phi_{\cal F},  {\cal R}_{\cal F} \big )$ 
is also a quasi-triangular quasi-bialgebra.
\end{pro} 
\begin{proof}
Let as consider the general twist ${\cal F} = \sum_j f_j \otimes g_j.$ In terms of the invertible elements $v:= \sum_{j}\epsilon(f_j) g_{j},$  $~w:= \sum_j \epsilon(g_j) f_j ,$ we have
\begin{equation}
(\epsilon \otimes \mbox{id}){\cal F} = v,\quad  (\mbox{id} \otimes \epsilon){\cal F} = w. \label{extra}
\end{equation}

Axioms  (1)-(2) of Definition \ref{definition1} as well as axioms of Definition \ref{definition2} are checked 
as in the original proof by Drinfeld (see also \cite{Drintw, Kassel}). For these the action of the counit on 
the twist is never used. The check of these axioms is somehow tedious, but nevertheless straightforward. 
Let us for instance check the axioms (1)-(3) of Definition \ref{definition2}:

\begin{itemize}
\item Axiom (1) of Definition \ref{definition2}: we compute
\begin{eqnarray}
\Delta^{(op)}_{\mathcal F}(a){\mathcal R}_{\mathcal F} &=& 
{\cal F}^{(op)}\Delta^{(op)}(a) ( {\cal F}^{(op)})^{-1}  {\cal F}^{(op)}{\cal R} {\cal F}^{-1} \nonumber\\ 
&=& {\cal F}^{(op)} {\cal R} \Delta(a) {\cal F}^{-1} = {\mathcal R}_{\mathcal F}\Delta_{\mathcal F}(a).
\end{eqnarray}
\item Axioms  (2) and (3) of Definition \ref{definition2}: we compute
\begin{equation}
(\mbox{id} \otimes\Delta_{\cal F}){\cal R}_{\cal F} = \big ( (\mbox{id} \otimes {\cal F}) (\mbox{id} \otimes \Delta){\cal F}^{(op)}\big )
 \big ((\mbox{id} \otimes \Delta){\cal R} \big )\big ( ( \mbox{id} \otimes \Delta){\cal F}^{-1}  (\mbox{id} \otimes {\cal F}^{-1} )\big ). \nonumber
\end{equation}
 It is convenient to  introduce some useful notation: 
\begin{equation}
{\cal F}_{1,23} :=(\mbox{id} \otimes \Delta){\cal F}, ~~~~~{\cal F}_{12,3} :=(\Delta \otimes \mbox{id}){\cal F},
\end{equation}
 and by the quasi-bialgebra axioms 
${\cal F}_{21,3} {\cal R}_{12} = {\cal R}_{12} {\cal F}_{12,3},$ $\  {\cal F}_{1,32} {\cal R}_{23} = {\cal R}_{23} {\cal F}_{1,23};$  also recall $({\cal R}_{\cal F})_{12} = {\cal F}_{21} {\cal R}_{12} {\cal F}_{12}^{-1}.$
We then denote according to the index notation  $(\mbox{id} \otimes \Delta){\cal F}^{(op)} = {\cal F}_{23,1},$ and we re-express  condition (\ref{fi}) as 
${\cal F}_{12} {\cal F}_{12,3} = {\cal F}_{23} {\cal F}_{1,23} \Phi_{123},$ then
\begin{eqnarray}
(\mbox{id} \otimes\Delta_{\cal F}){\cal R}_{\cal F} &=&
{\cal F}_{23}{\cal F}_{23,1}\big ( (\mbox{id} \otimes \Delta){\cal R}\big ) {\cal F}_{1,23}^{-1}{\cal F}_{23}^{-1}\nonumber\\
 &=& (\Phi_{\cal F})^{-1}_{231}{\cal F}_{31}{\cal F}_{2,31}\Phi_{231}  \big (  \Phi_{231}^{-1}  {\cal R}_{13} 
\Phi_{213} {\cal R}_{12} \Phi^{-1}_{123}\big ){\cal F}_{1,23}^{-1}{\cal F}_{23}^{-1} \nonumber\\
&=&  (\Phi_{\cal F})^{-1}_{231}  ({\cal R}_{\cal F})_{13}{\cal F}_{13}{\cal F}_{2,13} \Phi_{213} {\cal R}_{12} {\cal F}_{123}^{-1} \nonumber\\
&=&  (\Phi_{\cal F})^{-1}_{231}({\cal R}_{\cal F})_{13}  (\Phi_{\cal F})^{-1}_{213}{\cal F}_{21}{\cal F}_{21,3} {\cal R}_{12} {\cal F}_{123}^{-1} \nonumber\\
&=&  (\Phi_{\cal F})^{-1}_{231}({\cal R}_{\cal F})_{13} (\Phi_{\cal F})^{-1}_{213} ({\cal R}_{\cal F})_{12}{\cal F}_{12}{\cal F}_{12,3}
\Phi^{-1}_{123} {\cal F}_{1,23}^{-1}{\cal F}_{23}^{-1}\nonumber\\
&=& (\Phi_{\cal F})^{-1}_{231} ({\cal R}_{\cal F})_{13} (\Phi_{\cal F})^{-1}_{213} ({\cal R}_{\cal F})_{12} (\Phi_{\cal F})^{-1}_{123}.\label{basic1}
\end{eqnarray}

Similarly, for $(\Delta_{\cal F} \otimes \mbox{id}){\cal R}_{\cal F}.$
\end{itemize}

We shall now examine axioms (3)-(4) of Definition \ref{definition1},  and  working out 
the action of the counit on the twisted ${\cal R}$-matrix.  
\begin{itemize}
\item Axiom (3) of Definition \ref{definition1}: we recall (\ref{extra}) and axiom (3) of Definition \ref{definition1}, then
$(\epsilon \otimes \mbox{id})\Delta_{\cal F}(a) = (\epsilon \otimes \mbox{id}) \big ({\cal F} \Delta(a) {\cal F}^{-1} \big )  
=( c^{\cal F}_l)^{-1}  a c^{\cal F}_l ,$
where $c^{\cal F}_l = c_l v^{-1}$. Similarly, $(\mbox{id} \otimes \epsilon)\Delta_{\cal F}(a)=  (c^{\cal F}_r)^{-1} a c^{\cal F}_r, $ 
where $c^{\cal F}_r = c_r w^{-1}$, $\forall a \in {\cal A}.$
\item Axiom (4) of Definition \ref{definition1}: we recall that $(\mbox{id} \otimes \epsilon \otimes \mbox{id}) (1\otimes {\cal F}) = 1 \otimes v,$
$(\mbox{id} \otimes \epsilon \otimes \mbox{id}) ({\cal F}\otimes 1) = w \otimes 1,$ and  $(\mbox{id} \otimes \epsilon \otimes \mbox{id}) 
(\mbox{id}\otimes\Delta) {\cal F}= (1\otimes c_l^{-1}) {\cal F} (1 \otimes c_l),$ $(\mbox{id} \otimes \epsilon \otimes \mbox{id})
 (\Delta \otimes \mbox{id}) {\cal F}= (c_r^{-1} \otimes 1) {\cal F} ( c_r \otimes 1),$
bearing also in mind (\ref{fi}) and the fact that $(\mbox{id} \otimes \epsilon \otimes \mbox{id})\Phi= c_r \otimes  c_l^{-1},$  we deduce that 
$(\mbox{id} \otimes \epsilon \otimes \mbox{id})\Phi_{\cal F} = c^{\cal F}_r \otimes (c^{\cal F}_l)^{-1},$ 
and this concludes our proof as all the axioms of the quasi-triangular quasi-bialgebra are satisfied. 

\end{itemize}

We also examine $(\mbox{id} \otimes \epsilon){\cal R}_{\cal F} $ and $(\epsilon \otimes \mbox{id}) {\cal R}_{\cal F},$ recalling (\ref{R}) and (\ref{extra}) 
we conclude that $(\mbox{id} \otimes \epsilon){\cal R}_{\cal F} = v \big ((\mbox{id} \otimes \epsilon){\cal R}\big ) w^{-1}$ and 
$(\epsilon \otimes \mbox{id}) {\cal R}_{\cal F} =w \big ((\epsilon \otimes \mbox{id}){\cal R} \big )v^{-1}.$ When one of the quasi-bialgebras is a bialgebra,
e.g. set $ \Phi_{\cal F}= 1\otimes 1 \otimes 1$  ($c^{\cal F}_l = c^{\cal F}_r=1$), then 
$(\epsilon \otimes \mbox{id}) {\cal R}_{\cal F}=(\mbox{id} \otimes \epsilon){\cal R}_{\cal F} = 1,$ and 
consequently  $(\mbox{id} \otimes \epsilon){\cal R} = v^{-1}w$ and $(\epsilon \otimes \mbox{id}){\cal R}= w^{-1}v.$ 
Recalling also that  $c^{\cal F}_l = c_l v^{-1}$ and  $c^{\cal F}_r = c_r w^{-1}$
we deduce that $c_l=v$  and $c_r=w.$

When also $\Phi =1\otimes 1 \otimes 1,$ i.e. we are dealing 
with two bialgebras then we also have $(\epsilon \otimes \mbox{id}) {\cal R}=(\mbox{id} \otimes \epsilon){\cal R}= 1$ and hence $v=w,$ and 
due to Axiom (3) of (quasi) bialgebras we immediately deduce that $u =1.$ 
\end{proof}

We note that in the case of a triangular quasi-bialgebra the extra condition ${\cal R}^{(op)}{\cal R} = 1_{{\cal A} \otimes {\cal A}}$  holds, 
and due to (\ref{R}) we deduce that
${\cal R}_{\cal F}^{(op)}{\cal R}_{\cal F} = 1_{{\cal A} \otimes {\cal A}},$ so triangularity is also preserved. 

We examine in the next Corollary a special case of the Proposition \ref{twist}, 
namely the case where the Drinfeld twist is a special pure tensor (see also \cite{Drintw}, Dinfeld 1989).
\begin{cor}{\label{trivial}}
Let $\big ({\cal A},\Delta, \epsilon,\Phi, {\cal R} \big )$ be a 
quasi-triangular quasi-bialgebra and consider the special twist ${\cal F}=c_r \otimes c_l,$ 
where $ c_r$ and $c_l $ are derived from $\Phi$ according to Remark \ref{remco}.
Then the twisted quasi-triangular quasi-bialgebra $\big ({\cal A},\Delta_{\cal F}, \epsilon, \Phi_{\cal F},{\cal R}_{\cal F} \big )$ \cite{Drintw}, has trivial unit constraints i.e.  $c^{\cal F}_l=  c^{\cal F}_r = 1_{\cal A}.$
\end{cor}
\begin{proof}
For any quasi-triangular quasi-bialgebra 
$(\epsilon \otimes id){\cal R}= c_r^{-1} c_l$ and
$(id \otimes \epsilon){\cal R}= c_l^{-1} c_r.$
According to Remark \ref{remco} any quasi-triangular quasi-bialgebra can be suitably counit- normalized 
such that $\epsilon(c_r) = \epsilon(c_l)=1.$ 
We  now consider the special  twist
${\cal F} =c_r \otimes c_l$, which in the context of quasi-triangular quasi-bialgebras is  admissible.  
The ${\cal F}$-twisted quasi-triangular quasi-bialgebra has an $R$-matrix:
${\cal R}_{\cal F} = (c_l \otimes c_r) {\cal R} (c_r^{-1} \otimes c_l^{-1}),$
which satisfies: $(\epsilon \otimes \mbox{id}){\cal R}_{\cal F}= (\mbox{id} \otimes \epsilon){\cal R}_{\cal F} = 1$. 
The latter statement together with axiom (3) of Definition \ref{definition1}  lead to the ``trivialization" 
of the quasi-bialgebra, i.e. $c^{\cal F}_l = c^{\cal F}_r = 1_{\cal A}.$
\end{proof}

Note that twists can be in general composed, i.e. if the twist ${\cal F}$ sends quasi-bialgebra ${\cal A}$ to 
quasi-bialgebra ${\cal B}$,  with trivial unit constraints, and the twist ${\cal G}$ sends ${\cal B}$ to a third quasi-bialgebra ${\cal C}$, 
then the product ${\cal G} {\cal F}$ is a twist as well, sending ${\cal A}$ to ${\cal C}$.  This type of  factorization 
might make the technicalities of describing certain admissible twists a bit easier to follow, breaking it up into possibly more manageable building blocks.
In particular, if ${\cal A}$ is an arbitrary quasi-triangular quasi-bialgebra and ${\cal C}$ is a quasi-triangular bialgebra, 
then there always exist a quasi-triangular quasi-bialgebra ${\cal B}$ with trivial unit constraints, and twists ${\cal F}$ and ${\cal G}$ such that:
\begin{enumerate}
\item ${\cal F}$  is as described in Corollary \ref{trivial},
\item ${\cal G}$  satisfies $(\epsilon \otimes id) {\cal G}  = (id \otimes \epsilon) {\cal G} = 1$,
\item ${\cal B}$  is obtained from ${\cal A}$  by twisting with ${\cal F}$,
\item ${\cal C}$ is obtained from ${\cal B}$  by twisting with ${\cal G}$.
\end{enumerate}

$ $

We shall be focusing henceforth  on situations of twisting between a (quasi-)bialgebra  and a bialgebra in accordance to 
Propositions \ref{twist} for the special case of (quasi-)bialgebras of \ref{lemma1}, 
specifically $\Phi_{\cal F}= 1\otimes 1 \otimes 1$  and  $\Phi \neq 1\otimes1 \otimes 1.$ 
We are considering thus two specific cases that consist 
the appropriate framework to describe the quantum algebras \cite{DoiSmo1, DoiSmo2} emerging from
set-theoretic solutions of the Yang-Baxter equation expressed as suitable Drinfeld twists \cite{Doikoutw} 
as will be discussed in the subsequent section. 
We introduce now some useful notation and a worked out example  appropriate for our frame here when examining quantum 
algebras emerging form set-theoretic solutions of the YBE and their $q$-analogues (Sections 2 and 3), compatible also with the analysis in \cite{Doikoutw}.

\begin{rem}{\label{remdr}} According to  Proposition \ref{lemma1} we distinguish two cases:
Let $\big ({\cal A},\Delta, \epsilon,\Phi, {\cal R} \big )$ be a 
quasi-triangular quasi-bialgebra and $\big ({\cal A},\Delta_{\cal F}, \epsilon, {\cal R}_{\cal F} \big )$ a quasi-triangular bialgebra  
and let the conditions of 
Proposition \ref{lemma1} hold. We first recall the useful notation introduced in the proof of Proposition \ref{twist}: ${\cal F}_{1,23} :=
(\mbox{id} \otimes \Delta){\cal F}$ and 
${\cal F}_{12,3} :=(\Delta \otimes \mbox{id}){\cal F},$ and by the quasi-bialgebra axioms 
${\cal F}_{21,3} {\cal R}_{12} = {\cal R}_{12} {\cal F}_{12,3},$ $\  {\cal F}_{1,32} {\cal R}_{23} = {\cal R}_{23} {\cal F}_{1,23}. $
\begin{enumerate}
\item  If the associator satisfies $\Phi_{213} {\cal R}_{12} = {\cal R}_{12}\Phi_{123},$  we denote (index notation):
\begin{equation}
{\cal F}^*_{12,3} :=\big ((\Delta \otimes \mbox{id}){\cal F} \big )\Phi^{-1},
\end{equation}
then condition (\ref{fi}) can be re-expressed as ${\cal F}_{123} := {\cal F} _{23}{\cal F}_{1,23} = {\cal F}_{12} {\cal F}_{12,3}^*.$
Due to constraint (\ref{2}) we also  deduce ${\cal F}^*_{21,3} {\cal R}_{12} = {\cal R}_{12} {\cal F}^*_{12,3}.$ 
This  is compatible also with the first part of Proposition \ref{lemma1}.

\item If the associator satisfies $\Phi_{132} {\cal R}_{23} = {\cal R}_{23}\Phi_{123},$ we denote
\begin{equation}
{\cal F}^*_{1,23} =\big ( (\mbox{id} \otimes \Delta){\cal F}\big )\Phi,
\end{equation}
then condition (\ref{fi}) is re-expressed as ${\cal F}_{123} := {\cal F} _{23}{\cal F}^*_{1,23} = {\cal F}_{12} {\cal F}_{12,3}.$ 
In this case, due to constraint (\ref{2}), we also  deduce 
$ {\cal F}^*_{1,32} {\cal R}_{23} = {\cal R}_{23} {\cal F}^*_{1,23}.$ 
This is compatible with the second part of Proposition \ref{lemma1}.
\end{enumerate} 
\end{rem}

\begin{rem}{\label{remdr2}} Under the conditions of Proposition \ref{lemma1} the universal 
${\cal R}$-matrix satisfies the usual YBE, and hence the object 
${\mathrm T}_{1,23} := {\cal R}_{13} {\cal R}_{12}$ satisfies the RTT relation \cite{FadTakRes}:\\ 
${\cal R}({\mathrm T} \otimes  1) (1\otimes {\mathrm T}) =( 1 \otimes {\mathrm T}) ({\mathrm T} \otimes 1){\cal R}$.  
Also, due to the properties described in Remark \ref{remdr}:
\begin{equation}
({\mathrm T}_{\cal F})_{1,23} :=  ({\cal R}_{\cal F})_{13} ({\cal R}_{\cal F})_{12}= {\cal F}_{231} {\mathrm T}_{1,23} {\cal F}_{123}^{-1}.
\end{equation}
The above is checked as in \cite{Doikoutw} (the $N$-generalization also holds as described in \cite{Doikoutw}). 
\end{rem}

Before we focus on set-theoretic solutions of the YBE we will first work out a particular example
compatible with the special case of Remark \ref{specialcase}.
\begin{exa}{\label{cor1L}}
Let $\big ({\cal A}, \Delta_{{\cal A}}, \epsilon, {\cal R}_{\cal A}\big )$ be a quasi-triangular bialgebra.
Let also $u \in {\cal A}\otimes {\cal A}$ be a group-like element, i.e. $u$ is invertible and 
$\Delta_{\cal A}(u) = u\otimes u,$ and consequently $\epsilon(u) =1,$ $~\big [\Delta_{\cal A}(u),\ {\cal R}_{\cal A}\big ] = 0.$
Consider also the Drinfeld twist ${\cal F} = 1 \otimes u:$
${\cal R}= {\cal F^{(op)}}^{-1} {\cal R}_{\cal A} {\cal F}$ and $\Delta(a)= {\cal F}^{-1}
\Delta_{\cal A}(a) {\cal F},$ $a\in {\cal A}.$

Then $\big ({\cal A}, \Delta, \epsilon, \Phi,  {\cal R}\big )$ is a quasi-triangular quasi-bialgebra:
\begin{itemize}
\item $\Phi = 1\otimes 1 \otimes u^{-1}.$
\item $ c_r= 1,\ c_l = u.$
\item $(\Delta \otimes \mbox{id}) {\cal R} = {\cal R}_{13} {\cal R}_{23} \Phi_{123}$ and 
$(\mbox{id} \otimes \Delta)= \Phi^{-1}_{231}{\cal R}_{13} {\cal R}_{12}.$ 
\end{itemize}

We recall the twist:  ${\cal F} = 1 \otimes u,$
we can then readily check that 
$(\epsilon \otimes\mbox{id}){\cal  F} =u,$ 
and via $\epsilon(u)=1,$  we conclude $(\mbox{id} \otimes \epsilon)  {\cal F} = 1.$
Even though the condition $(\epsilon \otimes\mbox{id})  {\cal F} = 1,$ 
is now relaxed, the twist is still  admissible, indeed this is a rather trivial statement: ${\cal F}$ satisfies the following 
\begin{equation}
 {\cal F}_{123}:= {\cal F}_{12} {\cal F}_{12,3}^* = {\cal F}_{23} {\cal F}_{1,23} 
\end{equation}
where
${\cal F}_{1,23} := (\mbox{id} \otimes \Delta){\cal F} = 1 \otimes u \otimes u$ and ${\cal F}^*_{12,3} :=
\big ( (\Delta \otimes \mbox{id}){\cal F}\big )\Phi^{-1} = 1\otimes 1\otimes u^2 .$
The $N$-fold twist is then defined as
\begin{equation}
{\cal F}_{12...N}:= {\cal F}_{23..N} {\cal F}_{1,23..N} = {\cal F}_{12...N-1} {\cal F}_{12...N-1,N},
\end{equation}
more specifically ${\cal F}_{12..N} = 1 \otimes u \otimes u^2\otimes \ldots \otimes u^{N-1}.$
The twist is admissible and hence the YBE is also satisfied by ${\cal R}$  (see also \cite{Drin} and Proposition \ref{twist}).

Notice also that $\Phi_{213} {\cal R}_{12} ={\cal R}_{12} \Phi_{123},$ so Proposition \ref{lemma1}  
holds as well as Proposition \ref{twist} and Remark \ref{remdr}, 
then all the axioms of Definitions \ref{definition1} and \ref{definition2} are satisfied, indeed:
\begin{itemize}
\item $(\mbox{id} \otimes \Delta) \Delta(a) = (1\otimes 1\otimes u^{-1}) 
\big ((\Delta \otimes \mbox{id})\Delta(a) \big )(1\otimes 1 \otimes u),$  
$\forall a \in {\cal A}.$
\item Axiom 2 of Definition \ref{definition1} is trivially satisfied.
\item $(\epsilon \otimes \mbox{id}) \Delta(a) = u^{-1} a u$ and  $(\mbox{id} \otimes \epsilon) \Delta(a) =a$
\item $(\mbox{id}\otimes \epsilon \otimes \mbox{id})\Phi = 1 \otimes u^{-1}.$
\item $\Delta^{(op)}(a){\cal R} = {\cal R} \Delta(a),$ $\forall a \in {\cal A}.$
\item $(\mbox{id} \otimes \Delta){\cal R} = \Phi_{231}^{-1} {\cal R}_{13}{\cal R}_{12}\ $ and 
$\ (\Delta\otimes \mbox{id}){\cal R} = {\cal R}_{13}{\cal R}_{12} \Phi_{123}.$
\end{itemize}
And due to Proposition \ref{lemma1} the ${\cal R}$-matrix satisfies the Yang-Baxter equation. It is also straightforward to show, 
using that $\epsilon$ is an algebra homomorphism, that $(\epsilon \otimes \mbox{id}){\cal R} =u,\ (\mbox{id} \otimes \epsilon){\cal R} = u^{-1},$ 
as expected due to Proposition \ref{lemma1}. In this example it is clear that in the quasi-bialgebra setting Drinfeld twists do not require a counit constraint.
\end{exa}

\section{Quasi-bialgebras from Yangians}

\noindent In this section we focus on the quantum algebras emerging from involutive, non-degenerate,
set-theoretic solutions of the YBE.  From now on we work over the field $k=\mathbb{C}.$
It was shown in \cite{Doikoutw} after identifying a suitable  admissible twist that Baxterized set-theoretic solutions are always coming from the $\mathfrak{gl}_n$
Yangian $R$-matrix via the mentioned twist. We will show here that the respective  twisting of the $\mathfrak{gl}_n$ Yangian leads to a
quasi-triangular quasi-bialgebra in accordance to the findings of section 1.

\subsection{Set-theoretic solutions of the YBE}

\noindent We present in this section basic background information regarding set-theoretic solutions of the 
Yang-Baxter equation as well as a brief review on the recent findings of \cite{DoiSmo1}
on the links between set-theoretic solutions of the Yang-Baxter equation and quantum algebras.

\noindent Let $X=\{{\mathrm x}_{1},{\mathrm x}_{2}, \ldots, {\mathrm x}_n\}$ be a finite set and ${\check r}: 
X\times X\rightarrow X\times X,$ such that
 \[{\check r}(x,y)= \big (\sigma _{x}(y), \tau _{y}(x)\big ).\] 
We say that $\check r$ is non-degenerate if $\sigma _{x}$ 
and $\tau _{y}$ are bijective functions. Also, the solution $(X, \check r)$  is involutive: $\check r ( \sigma _{x}(y), \tau _{y}(x)) = 
(x, y)$, ($ \check r^2 (x, y)) = (x, y)$). We focus on non-degenerate,  involutive solutions of the set-theoretic braid equation:
\[({\check r}\times id_{X})(id_{X}\times {\check r})({\check r}\times id_{X})=(id_{X}\times {\check r})({\check r}
\times id_{X})(id_{X}\times {\check r}).\]

Let $V$ be the space of dimension equal to the cardinality of $X$, and with a slight abuse of notation, 
let  $\check r$ also  denote the matrix associated to the linearisation of ${\check r}$ on 
$V={\mathbb C }X$ (see \cite{LAA} for more details), i.e.
$\check r$  is the $n^2 \times n^2$ matrix: 
\begin{equation}
{\check r}=\sum_{x, y\in X} e_{x, \sigma_x(y)}\otimes e_{y, \tau_y(x)}. \label{brace1}
\end{equation}
where $e_{x, y}$ is the $n \times n$  matrix: $(e_{x,y})_{z,w} =\delta_{x,z}\delta_{y,w} $.
 The matrix $\check r:V\otimes V\rightarrow V\otimes V$ satisfies as expected the (constant) Braid equation:
\[({\check r}\otimes I_{V})(I_{V}\otimes {\check r})({\check r}\otimes I_{V})=(I_{V}\otimes {\check r})({\check r}\otimes 
I_{V})(I_{V}\otimes {\check r}),\]
where $I_V$ the identity matrix.  Note  also that $\check r$ is involutive, i.e.  ${\check r}^{2}=I_{V \otimes V}.$ 

We define also, $r={\mathcal P}{\check r}$, where ${\mathcal P} = \sum_{x, y \in X} e_{x,y} \otimes e_{y,x}$ 
is the permutation operator; consequently
${r}=\sum_{x, y\in X} e_{y,\sigma_x(y)}\otimes e_{x, \tau_y(x)}.$
The Yangian \cite{Yang} is a special case: $\check r= \sum_{x, y \in X} e_{x,y} \otimes e_{y,x}.$

$ $

We recall now the Yang-Baxter equation in the braid form in the presence of spectral parameters 
$\lambda_1,\ \lambda_2$ ($\delta = \lambda_1 - \lambda_2$) :
\begin{equation}
\check R_{12}(\delta)\ \check R_{23}(\lambda_1)\ \check R_{12}(\lambda_2) = \check R_{23}(\lambda_2)\
 \check R_{12}(\lambda_1)\ \check R_{23}(\delta) . \label{YBE1}
\end{equation}
where $\check R: V \otimes \to V \otimes V$  and let in general $\check R = \sum_{j} a_j \otimes b_j,$ then in the index notation 
$\check R_{12} =\sum_j a_j \otimes b_j \otimes I_V,$ 
 $\check R_{23} =\sum_j  I_V \otimes a_j \otimes b_j $ and  $\check R_{13} =\sum_j a_j \otimes  I_V \otimes b_j.$

We focus henceforth on involutive, non-degenerate set-theoretic  solutions of the YBE, given by (\ref{brace1}). 
The set-theoretic solution $\check r$ (\ref{brace1})
is a representation of the $A$-type Hecke algebra for $q=1$ (see also \cite{DoiSmo1}),
as $\check r$ satisfies the braid relations and also 
$\check r^2 = I_{X \otimes X}.$ The set-theoretic $\check r$ provides a representation 
of the $A$-type Hecke algebra,  hence
Baxterized solutions of the Yang-Baxter equation can be derived \cite{DoiSmo1}:
\begin{equation}
\check R(\lambda) = \lambda \check r + I_X\otimes I_X, \label{braid1}
\end{equation}
where $I_{X}$ is the identity matrix of dimension equal to the cardinality of the set $X$.
Let also $R(\lambda) = {\cal P} \check R(\lambda)$, (recall the permutation operator ${\cal P} = \sum_{x, y} e_{x,y}\otimes e_{y,x}$),  
then the following  properties for $R$-matrices coming from set-theoretic solutions were shown in \cite{DoiSmo1}:\\
\begin{enumerate}
\item Unitarity: $R_{12}(\lambda)\  R_{21}(-\lambda) = (-\lambda^2 +1)  I_{X\otimes X}.$
\item  Crossing Unitarity: $ R_{12}^{t_1}(\lambda)\ R_{12}^{t_2}(-\lambda -n) = 
\lambda(-\lambda -n) I_{X\otimes X}.$ 
\item $R_{12}^{t_1 t_2}(\lambda) = R_{21}(\lambda),$
where $^{t_{1,2}}$ denotes transposition on the first, second space respectively.
\end{enumerate}

We give a brief account on the quantum algebra emerging from Baxterized set-theoretic solutions of the Yang-Baxter equation.
Our approach on deriving the quantum group associated to set-theoretic solutions \cite{DoiSmo1, DoiSmo2}
is based on the FRT construction \cite{FadTakRes}, which is in a sense dual to the Hopf algebraic description \cite{Drinfeld}.
Given a solution of the Yang-Baxter equation, the quantum algebra is defined via the fundamental relation \cite{FadTakRes}(we have multiplied the familiar RTT relation with the permutation operator):
\begin{equation}
\check R_{12}(\lambda_1 -\lambda_2)\ L_1(\lambda_1)\ L_2(\lambda_2) = L_1(\lambda_2)\ L_2(\lambda_1)\ 
\check R_{12}(\lambda_1 -\lambda_2),  \label{RTT}
\end{equation}
where $\check R(\lambda) \in \mbox{End}({\mathbb C}^n) \otimes  \mbox{End}({\mathbb C}^n)$, $\ L(\lambda) \in 
\mbox{End}({\mathbb C}^n ) \otimes {\cal A}$ and ${\cal A}$ is the quantum algebra defined by (\ref{RTT}). 
We focus on solutions given by (\ref{braid1}), (\ref{brace1}). 
The defining relations of the corresponding quantum algebra were derived in \cite{DoiSmo1}:

\noindent The quantum algebra associated to the set-theoretic $R$-matrix  (\ref{braid1}), (\ref{brace1}) 
is defined by generators $L^{(m)}_{zw},\ z, w \in X$, and defining relations
\begin{eqnarray}
L_{z,w}^{(n)} L_{\hat z, \hat w}^{(m)} - L_{z,w}^{(m)} L_{\hat z, \hat w}^{(n)} &=& 
L^{(m)}_{z, \sigma_w(\hat w)} L^{(n+1)}_{\hat z,\tau_{\hat w}( w)}- L^{(m+1)}_{z, \sigma_w(\hat w)} 
L^{(n)}_{\hat z, \tau_{\hat w}( w)}\nonumber\\ &-& L^{(n+1)}_{ \sigma_z(\hat z),w} 
L^{(m)}_{\tau_{\hat z}( z), \hat w }+ L^{(n)}_{ \sigma_z(\hat z, )w}  L^{(m+1)}_{\tau_{\hat z}( z), \hat w}. \label{fund2}
\end{eqnarray}

The proof is based on the fundamental relation (\ref{RTT}) and the form of the Baxterized set-theoretic 
$R$-matrix
(for the detailed proof see \cite{DoiSmo1}). Recall also that
in the index notation we define $\check R_{12} = \check R \otimes 1_{\cal A}$:
\begin{eqnarray}
&&  L_1(\lambda) = \sum_{z, w \in X} e_{z,w} \otimes I_X \otimes L_{z,w}(\lambda),\ \quad  L_2(\lambda)= \sum_{z, w \in X}
I_X  \otimes  e_{z,w}  \otimes L_{z,w}(\lambda)  \label{def}
\end{eqnarray} where $L_{z,w}(\lambda) = \sum_{m=0}^{\infty}\lambda^{-m}L_{z,w}^{(m)}$ and $L_{z,w}^{(m)}$
are the generators of the affine algebra ${\cal A}$ and $\check R$ is given in (\ref{braid1}), (\ref{brace1}). 
Note that the element
$ {\mathrm T}_{1,23}(\lambda)= L_{13}(\lambda) L_{12}(\lambda),$  also satisfies (\ref{RTT})  \cite{FadTakRes, Drinfeld}, i.e it is 
 a tensor representation of the quantum algebra.

\subsection{The Yangian ${\cal Y}(\mathfrak{gl}_n)$} 

\noindent We now present a brief review on the $\mathfrak{gl}_n$ Yangian useful for our purposes here.
We first briefly recall the definition of the $\mathfrak{gl}_n$ Yangian, 
which is most relevant in our present investigation (for a review on Yangians see e.g. \cite{Drinfeld, ChaPre, yangians}). 
We also review the Yangian as a Hopf algebra and we then comment on the action of the antipode after a suitable twist of the algebra.

\begin{defn}{\label{Yangian}}
The $\mathfrak{gl}_n$ Yangian ${\cal Y}(\mathfrak{gl}_n)$,
is a non-abelian algebra with generators ${\mathrm Q}_{ab}^{(p)}$, $p\in \big \{ 1,2,\ldots\big\}$, $ a, b \in \big  \{ 1,\ 2, \dots, n\big  \}$
and defining relations given below 
\ba
&&\Big [ {\mathrm
Q}_{ab}^{(1)},\ {\mathrm Q}_{cd}^{(1)} \Big ] =\delta_{cb}{\mathrm Q}^{(1)}_{ad} - \delta_{ad}{\mathrm Q}^{(1)}_{cb}
\cr && \Big [ {\mathrm Q}_{ab}^{(1)},\ {\mathrm Q}_{cd}^{(2)} \Big ] =\delta_{cb}
{\mathrm Q}^{(2)}_{ad} - \delta_{ad}{\mathrm Q}^{(2)}_{cb}
\cr
  && \Big [ {\mathrm Q}_{ab}^{(2)},\ {\mathrm Q}_{cd}^{(2)} \Big ]
=\delta_{cb}{\mathrm Q}^{(3)}_{ad} -\delta_{ad}{\mathrm Q}^{(3)}_{cb}
+{1\over 4}{\mathrm Q}_{ad}^{(1)}(\sum_{e} {\mathrm
Q}_{ce}^{(1)}{\mathrm Q}_{eb}^{(1)})- {1\over 4}(\sum_{e}{\mathrm Q}_{ae}^{(1)}{\mathrm Q}_{ed}^{(1)})
{\mathrm Q}_{cb}^{(1)} \label{A}
\ea
also relations
\ba
&& \Big [ {\mathrm Q}_{ab}^{(1)},\ \Big [ {\mathrm Q}_{cd}^{(2)},\ {\mathrm Q}_{ef}^{(2)} \Big ] \Big ]- \Big
[ {\mathrm Q}_{ab}^{(2)},\
\Big [{\mathrm Q}_{cd}^{(1)},\ {\mathrm Q}_{ef}^{(2)} \Big ] \Big ] =\nonumber\\ &&   
{1\over 4}\sum_{p,q} \Big ( \Big [ {\mathrm Q}_{ab}^{(1)},\ \Big [ {\mathrm Q}_{cp}^{(1)}{\mathrm Q}_{pd}^{(1)},\
{\mathrm Q}_{eq}^{(1)}{\mathrm
Q}_{qf}^{(1)} \Big ] \Big ] -\Big [ {\mathrm Q}_{ap}^{(1)}{\mathrm Q}_{pb}^{(1)},\ \Big [ {\mathrm Q}_{cd}^{(1)},\ 
{\mathrm Q}_{eq}^{(1)}{\mathrm Q}_{qf}^{(1)} \Big ] \Big ]  \Big ), \\
& &\ldots \mbox{(higher orders)} \nonumber
\ea
\end{defn}

It will be useful for what follows to introduce at this point the definition of a quasi-Hopf algebra \cite{Drinfeld}.

\begin{defn}{\label{definition1b}} A quasi-Hopf algebra ${\cal A}$ is a quasi-bialgebra $\big ( {\cal A} ,\Delta  , \epsilon, \Phi \big )$ 
for which there exist  $\alpha ,\beta \in {\mathcal {A}}$ and a bijective algebra anti-homomorphism $S: {\cal A} \to {\cal A}$ (the antipode) such that
\begin{enumerate}
\item $\sum_j S(f_j) \alpha h_j = \epsilon(w) \alpha\ $ and $\ \sum_j f_j\beta S(h_j) = \epsilon(w) \beta$\\ where 
$\Delta(w) = \sum_j f_j \otimes h_j,$ $\forall w \in {\cal A}.$

\item $\sum_j x_j\beta S(y_j) \alpha z_j = 1_{\cal A}\ $ and $\ \sum_j S(\hat x_j)\alpha \hat y_j \beta S(\hat z_j) = 1_{\cal A}$\\
where $\Phi = \sum_j x_j \otimes y_j \otimes z_j$ and $\Phi^{-1} = \sum_j \hat x_j \otimes \hat y_j \otimes \hat z_j.$
\end{enumerate}
\end{defn}
Quasi-Hopf algebras generalize Hopf algebras in the same way that quasi-bialgebras generalize bialgebras.

The Yangian $\big ( {\cal Y}(\mathfrak{gl}_n) , \Delta_Y, \epsilon, S_Y, {\cal R}_{Y}\big )$  is 
a quasi-triangular Hopf algebra over ${\mathbb C}$ equipped with:
\begin{itemize}
\item A coproduct
$\Delta_Y: {\cal Y}(\mathfrak{gl}_n) \rightarrow {\cal Y}(\mathfrak{gl}_n) \otimes {\cal Y}(\mathfrak{gl}_n)$ such that
\ba
\Delta_Y({\mathrm Q}_{ab}^{(1)}) &=& {\mathrm Q}_{ab}^{(1)} \otimes 1 + 1 \otimes
{\mathrm Q}_{ab}^{(1)} \cr
\Delta_Y({\mathrm Q}_{ab}^{(2)}) &=& {\mathrm Q}_{ab}^{(2)}
\otimes 1 + 1 \otimes {\mathrm Q}_{ab}^{(2)} + {1 \over 2}
\sum _{d=1}^n({\mathrm Q}_{ad}^{(1)}\otimes
{\mathrm Q}_{db}^{(1)}-
{\mathrm Q}_{db}^{(1)}\otimes {\mathrm Q}_{ad}^{(1)}). \label{cop}
\ea
The $l$-fold co-product  ($l$ is an integer greater than 2)  $\Delta_Y^{(l)}: {\cal
Y}(\mathfrak{gl}_n) \ \to \ {\cal Y}^{\otimes (l)}(\mathfrak{gl}_n) $ is defined as
\be
\Delta_Y^{(l)} = (\mbox{id} \otimes \Delta_Y^{(l-1)})\Delta_Y =(\Delta_Y^{(l-1)} \otimes \mbox{id}) 
\Delta_Y. \label{cop22}
\ee

 \item A co-unit $\epsilon: {\cal Y}(\mathfrak{gl}_n) \to {\mathbb C}$,  such that
\begin{equation}
\epsilon({\mathrm Q}^{(1)}_{ab}) = \epsilon({\mathrm Q}^{(2)}_{ab}) =0,
\end{equation}
\item An antipode $S_Y: {\cal Y}(\mathfrak{gl}_n) \to {\cal Y}(\mathfrak{gl}_n)$ such that
\begin{equation}
S_Y({\mathrm Q}^{(1)}_{ab}) = - {\mathrm Q}_{ab}^{(1)}, \qquad S_Y({\mathrm Q}^{(2)}_{ab}) =
 - {\mathrm Q}^{(2)}_{ab} + {1 \over 2} {\mathrm Q}^{(1)}_{ab}
\end{equation}

\item Also, there exists an invertible element ${\cal R}_Y\in {\cal Y}(\mathfrak{gl}_n) \otimes {\cal Y}(\mathfrak{gl}_n)$ 
(the universal ${\cal R}$-matrix), such that
\begin{enumerate}
\item ${\cal R}_Y \Delta_Y({\mathrm Q}_{ab}^{(p)}) = \Delta_Y^{(op)}({\mathrm Q}_{ab}^{(p)})  {\cal R}_Y.$
\item $(\Delta_Y \otimes \mbox{id}){\cal R}_Y =({\cal R}_Y)_{13} ({\cal R}_Y)_{23} $ and $(\mbox{id} 
\otimes \Delta_Y){\cal R}_Y= ({\cal R}_Y)_{13} ({\cal R}_Y)_{12}$.
\end{enumerate}
\end{itemize}

We recall now Example \ref{cor1L} and we briefly discuss the notion of the antipode in the following:
\begin{rem}{\label{remanti}} We specialize Example \ref{cor1L} to the case of the Yangian 
${\cal Y}(\mathfrak{gl}_n)$, i.e. 
$\big( {\cal Y}(\mathfrak{gl}_n), \Delta, \epsilon, \Phi, {\cal R} \big ), $ where $\Phi = 1 \otimes 1 \otimes u^{-1}$ for some group-like element $u,$
is a quasi-triangular quasi-bialgebra.  We now want to test whether 
$\big ( {\cal Y}(\mathfrak{gl}_n),  \Delta, \epsilon,\Phi,  {\cal R}\big)$ is a quasi-Hopf algebra,  i.e.  if there exist $S$, $\alpha$, $\beta$ such that axioms (1) and (2) in Definitio \ref{definition1b} hold.

Suppose that the group-like element $u$ exists,  then $\Delta(u) = u \otimes u$ and:
\begin{enumerate}
\item $S(u)\alpha u =\alpha$ and $u\beta S(u) =\beta,$ which lead to $S(u) = \alpha u^{-1} \alpha^{-1} = \beta^{-1} u^{-1}\beta.$
\item $\beta \alpha u^{-1} =1_{Y}$ and $\alpha \beta S(u) = 1_Y,$ which lead to $u = \beta \alpha$ and $S(u) = \beta^{-1} \alpha^{-1}.$
\end{enumerate} 
All the above equations are self-consistent. Similarly, we can show that  $S(u^{-1}) = \alpha \beta$

We now check the axioms of Definition \ref{definition1b} for the primitive elements ${\mathrm Q}_{ab}^{(1)},$  with a coproduct after twisting given as
$\Delta({\mathrm Q}_{ab}^{(1)}) = {\mathrm Q}_{ab}^{(1)} \otimes 1 + 1 \otimes u^{-1}{\mathrm Q}_{ab}^{(1)}u $:
\begin{enumerate}
\item $S({\mathrm Q}_{ab}^{(1)}) \alpha + \alpha u^{-1}{\mathrm Q}_{ab}^{(1)} u =0$ and 
${\mathrm Q}_{ab}^{(1)}\beta + \beta S(u^{-1}{\mathrm Q}_{ab}^{(1)} u) = 0,$ which lead to (by requiring also that
 $S$ is an algebra anti-homomorphism): $S({\mathrm Q}_{ab}^{(1)}) = -\alpha u^{-1}{\mathrm Q}_{ab}^{(1)} u \alpha^{-1}$ and 
$S({\mathrm Q}_{ab}^{(1)}) = -\alpha {\mathrm Q}_{ab}^{(1)} \alpha^{-1}.$ 
The two latter expressions lead to $\big [{\mathrm Q}_{ab}^{(1)}, u\big ]=0,$ which is not true in general. 
\end{enumerate}

\noindent The axioms for the antipode restrict $u$ to be in the center of the algebra. This is very limiting, 
and in general is not true.  Indeed, a simple class of such non-central group-like elements can be defined as $u = e^{Q_{ab}^{(1)}},$ where recall $Q_{ab}^{(1)}$ are primitive elements of the algebra.
Also,  a specific example of a represented non-central $u$ is used in the next subsection 
for a special class of set-theoretic solutions.

To conclude, we are not able to define $S({\mathrm Q}_{ab}^{(1)})$ in a consistent way for a generic $u$  
by strictly following 
the axioms of Definition \ref{definition1b}.
A generalization on the axiomatic formulation for the antipode in quasi-Hopf algebras with nontrivial unit constraints may be in order, however this issue will 
be thoroughly studied elsewhere.
\end{rem}

It is useful for the purposes of the present investigation to introduce the evaluation representation
$\pi_{\lambda}: {\cal Y}(\mathfrak{gl}_n)\to \mbox{End}({\mathbb C}^n)$, $\lambda \in {\mathbb C},$ such that
\begin{equation}
\pi_{\lambda}({\mathrm Q}_{ab}^{(1)}) = e_{a,b},  \quad  \pi_{\lambda}({\mathrm Q}_{ab}^{(2)}) =  f_{a,b}
\end{equation}
where we define $f_{a,b} := \lambda e_{a,b}$.
Let $R_Y: V \otimes V\ to V \otimes V$ be the $R$-matrix associated to Yangian with explicit form
$R_Y(\lambda) = I^{\otimes 2} + {1\over \lambda} {\cal P},$ where
${\cal P}$ is the permutation operator and $I$ is the $n\times n$ identity matrix.

We also  introduce the following convenient notation:
\begin{equation}
\big (\pi_{\lambda_1} \otimes \pi_{\lambda_2}\big)\Delta_Y({\mathrm Q}_{ab}^{(1)}) = \Delta_Y(e_{a,b}), \qquad 
\big (\pi_{\lambda_1} \otimes \pi_{\lambda_2}\big)\Delta_Y({\mathrm Q}_{ab}^{(2)}) = \Delta_Y(f_{a,b}; \lambda_1, \lambda_2), \label{eval}
\end{equation}
also, $\big (\pi_{\lambda_1} \otimes \pi_{\lambda_2}\big)\Delta_Y^{(op)}({\mathrm Q}_{ab}^{(2)}) = \Delta_Y^{(op)}(f_{a,b}; \lambda_1, \lambda_2)$ and 
$\Delta_Y^{(op)}(f_{a,b}; \lambda_1, \lambda_2) = {\cal P} \Delta_Y(f_{a,b}: \lambda_2, \lambda_1){\cal P},$  
then we explicitly express the coproducts as
\begin{eqnarray}
&& \Delta_Y(e_{i,j}) = \Delta^{(op)}(e_{i,j})=e_{i,j} \otimes I + I \otimes e_{i,j}, \cr
&& \Delta_Y(f_{i,j}; \lambda_1, \lambda_2) = \lambda_1 e_{i,j} \otimes I + \lambda_2 I \otimes e_{i,j} + {1\over 2}\big(e_{i,k} \otimes e_{k,j} - 
e_{k,j} \otimes e_{i,j} \big ).\label{coY}
\end{eqnarray}
The Yangian $R$-matrix satisfies the following intertwining relations:
\begin{eqnarray}
&& R_Y(\lambda_1 -\lambda_2) \Delta_Y(e_{i,j})  = \Delta_Y(e_{i,j}) R_Y(\lambda_1 - \lambda_2) \cr
&& R_Y(\lambda_1 -\lambda_2) \Delta_Y(f_{i,j}; \lambda_1, \lambda_2)  =  \Delta_Y^{(op)}(f_{i,j}; \lambda_1, \lambda_2)
 R_Y(\lambda_1 - \lambda_2). \label{basicY}
\end{eqnarray}
Specifically, $\check R_Y ={\cal P} R_{Y}$ is $\mathfrak{gl}_n$ invariant, i.e.$\big [\check R_Y(\lambda),\ \Delta_{Y}(e_{x,y}) \big ] =0.$

And with this we conclude our short discussion on the $\mathfrak{gl}_n$ 
Yangian, which will be useful for the findings of the next subsection.

\subsection{Set-theoretic solutions of the YBE and quasi-bialgebras}

\noindent After the brief review on the $\mathfrak{gl}_n$ Yangian we may 
now move on to our main aim which is the study 
of set-theoretic solutions of the Yang-Baxter equation
associated to quasi-bialgebras.
We first review some fundamental results on the admissible Drinfeld twist for involutive
 set-theoretic solution of the YBE derived in \cite{Doikoutw} and we use these admissible 
twists to produce quasi-bialgebras associated to the Yangian.

$ $

From Proposition 3.3 in \cite{DoiSmo2} we can extract explicit forms for the twist 
$F \in\mbox{End}({\mathbb C}^n) \otimes\mbox{End}({\mathbb C}^n)  $ and state 
the following Proposition which is Proposition 3.10 in \cite{Doikoutw}.

\begin{pro}{\label{twistlocal} (\cite{DoiSmo2, Doikoutw})} 
Let $\check r = \sum_{x,y \in X} e_{x, \sigma_x(y)} \otimes e_{y, \tau_y(x) }$ be
 the set-theoretic solution of the braid YBE,  ${\cal P}$ is the permutation 
operator and $\hat V_k,\ V_k$ are their respective eigenvectors.  
Let $F^{-1} = \sum_{k=1}^{n^2} \hat V_k\  V_k^T$ 
be the similarity transformation (twist), such that $\check r = F^{-1} {\cal P} F.$
Then the twist can be explicitly expressed as 
$F = \sum_{x\in X} e_{x,x} \otimes {\mathbb V}_x$, where we define ${\mathbb V}_x =\sum_{y \in X} e_{\sigma_{x}(y), y}$. 
\end{pro}
For a detailed proof of the Proposition we refer the interested reader to \cite{DoiSmo2} and \cite{Doikoutw}.
However, by recalling that  $r = {\cal P} \check r,$ 
and using the fact that $\sigma_x,\ \tau_y$ 
are bijections, we confirm by direct computation that $(F^{(op)})^{-1} F= 
\sum_{x, \in X} e_{y, \sigma_x(y)} \otimes e_{x, \tau_y(x)}  =r$.  
This is an admissible twist as was shown in \cite{Doikoutw} and as will be 
discussed later in the text.
 
Let the Baxterized  solution of the YBE be $R(\lambda) = \lambda   r + {\cal P}.$ 
If $r$ satisfies the YBE and $r_{12} r_{21} =I$
then the Baxterized $R(\lambda)$ matrix also satisfies the YBE.
If $r= {\cal P} \check r$ is the set-theoretic solution of the YBE then, $R_{12}(\lambda)=F^{-1}_{21} (R_Y)_{12}(\lambda) F_{12}$, 
where $R_Y(\lambda) = \lambda I + {\cal P}$ is the Yangian $R$-matrix. 
This immediately follows from the form 
$R_Y(\lambda) = \lambda I+ {\cal P}$, 
and the property $F_{21}^{-1} {\cal P}_{12} {\cal F}_{12} = {\cal P}_{12}$. Note also  
that the twist is not uniquely defined, for instance an alternative twist is of the form 
$G= \sum_{x,y\in X}e_{\tau_{y}(x), x} \otimes e_{y,y} $, 
and $\sum_{x, \in X} e_{y, \sigma_x(y)} \otimes e_{x, \tau_y(x)}=  G_{21}^{-1} G_{12},$ see \cite{Doikoutw}.

$ $

Before we introduce the next fundamental quantities it is useful to prove the following Proposition. 
\begin{pro}{\label{basicl}} Let  $R_{Y}: {\mathbb C}^{n} \otimes {\mathbb C}^{n} \to 
{\mathbb C}^{n} \otimes {\mathbb C}^{n}$ be the Yangian $R$-matrix.
Let also ${\mathbb V}_{\eta} = \sum_{x \in X} e_{\sigma_{\eta}(x), x},$ $\forall \eta\in X$ consist a family of group-like elements, 
i.e. $\Delta_{Y}({\mathbb V}_{\eta}) ={\mathbb V}_{\eta} \otimes {\mathbb V}_{\eta}$ and $R(\lambda) = \lambda r +{\cal P},$
where $r$ is the set-theoretic solution $r = \sum_{x,y \in X} e_{y, \sigma_{x}(y)} \otimes e_{x, \tau_y(x)}.$ Then
\begin{equation}
\Delta^{(op)}({\mathbb V}_{\eta}) R(\lambda) = R(\lambda) \Delta({\mathbb V}_{\eta}),
\end{equation}
where $\Delta^{(op)}({\mathbb V}_{\eta}) = {\cal P} \Delta({\mathbb V}_{\eta}) {\cal P},$
\begin{equation}
\Delta({\mathbb V}_{\eta}) = \sum_{x\in X}e_{\sigma_{\eta}(x), x} \otimes \sum_{y \in X} e_{\sigma_{\tau_{x}(\eta)}(y), y}\big \vert_{C_1=0}, \label{twistdelta}
\end{equation}
and $C_1 = \sigma_{\sigma_{\eta}(x)}(\sigma_{\tau_{x}(y)}(y)) - \sigma_{\eta}(\sigma_x(y)).$
\end{pro}
\begin{proof}
By Proposition \ref{twistlocal} we have $R(\lambda) =( F^{(op)})^{-1} R_{Y}(\lambda) F,$
where we recall $F = \sum_{x\in X} e_{x,x} \otimes {\mathbb V}_x$. 
Recall also $\Delta_{Y}({\mathbb V}_{\eta}) = {\mathbb V}_{\eta} \otimes {\mathbb V}_{\eta},$ then
\begin{eqnarray}
&& \Delta_Y({\mathbb V}_{\eta}) R_Y(\lambda) = R_Y(\lambda)\Delta_{Y}({\mathbb V}_{\eta})\ \Rightarrow\  \cr
&&  (F^{(op)})^{-1}\big (  \Delta_Y({\mathbb V}_{\eta}) R_Y(\lambda)\big ) F =(F^{(op)})^{-1}
\big ( R_Y(\lambda)\Delta_{Y}({\mathbb V}_{\eta}) \big ) F\ \Rightarrow\cr
&&  \Delta^{(op)}({\mathbb V}_{\eta}) R(\lambda) = R(\lambda)\Delta({\mathbb V}_{\eta}),
\end{eqnarray}
where $\Delta({\mathbb V}_{\eta}) = F^{-1}\Delta_{Y}({\mathbb V}_{\eta}) F.$ By recalling that $F=\sum_{\eta} e_{\eta, \eta} 
\otimes {\mathbb V}_{\eta}$ and ${\mathbb V}_{\eta} = \sum_{x\in X} e_{\sigma_{\eta}(x), x}$  we obtain
\begin{eqnarray}
\Delta({\mathbb V}_{\eta}) &=& \big (\sum_{\bar \eta, \bar x\in X} e_{\bar \eta, \bar \eta} \otimes e_{\sigma_{\bar \eta}(\bar x), \bar x} \big ) 
\big ( \sum_{\xi \in X}e_{\sigma_{\eta}(\xi), \xi} \otimes \sum_{\zeta \in X}e_{\sigma_{\eta}(\zeta), \zeta}\big ) 
\big (\sum_{\hat \eta, \hat x\in X} e_{\hat \eta, \hat \eta} \otimes e_{\hat x, \sigma_{\hat\eta}(\hat x)}\big ) \cr
& = & \sum_{\xi, \bar x, \hat x \in X} e_{\sigma_{\eta}(\xi), \xi}\otimes e_{\bar x, \hat x}
\end{eqnarray}
subject to the following constraints: $\sigma_{\sigma_{\eta}(\xi)}(\bar x) = \sigma_{\eta}(\zeta)$ and $\zeta = \sigma_{\xi}(\hat x),$ which lead to 
\begin{equation}
\sigma_{\sigma_{\eta}(\xi)}(\bar x) = \sigma_{\eta}(\sigma_{\xi}(\hat x)). \label{cc}
\end{equation}
Also,  the constraint $C_1 =0$ holds for the twist $F$ to be admissible,  i.e. $\sigma_{\sigma_{\eta}(\xi)}(\sigma_{\tau_{\xi}(\eta)}(\hat x)) =  \sigma_{\eta}(\sigma_{\xi}(\hat x))$, 
which combined with (\ref{cc})  leads to $\bar x = \sigma_{\tau_{\xi}(\eta)}(\hat x),$ and thus
\begin{equation}
\Delta({\mathbb V}_{\eta}) = \sum_{x\in X}e_{\sigma_{\eta}(x), x} \otimes \sum_{y \in X} e_{\sigma_{\tau_{x}(\eta)}(y), y}\big \vert_{C_1=0}.
\end{equation}
Note that the constraint $C_1=0$  holds so that  the set-theoretic $r$-matrix satisfies the YBE (details on the constraints 
due to YBE in the form used here, see also e.g.  
\cite{Doikoutw} and relevant references therein).
\end{proof}

We recall that in Proposition 3.13 in \cite{Doikoutw} the following quantities were introduced:
\begin{eqnarray}
&&  F_{1,23} = \sum_{x, y, \eta \in X} e_{\sigma_{\eta}(x), \sigma_{\eta}(x)} \otimes e_{\eta, \tau_x(\eta)}\
 \otimes e_{\sigma_{x}(y), y}\big \vert_{C_1=0} \label{cop1}\\
&& F^*_{12,3} =\sum_{x, y, \eta \in X} e_{\sigma_{\eta}(x), \sigma_{\eta}(x)} \otimes e_{\tau_x(\eta), \tau_x(\eta) }\ 
\otimes e_{\sigma_{\eta}(\sigma_{x}(y)), y}\big \vert_{C_1=0} \label{cop2a}
\end{eqnarray}
where $C_1 = \sigma_{\sigma_{\eta}(x)}(\sigma_{\tau_{x}(y)}(y)) - \sigma_{\eta}(\sigma_x(y))$. 
Let also $\check r = \sum_{x, y \in X } e_{x, \sigma_{x}(y)} \otimes e_{y, \tau_{y}(x)}$, then 
\begin{equation}
 \check r_{12} F^*_{12,3} = F^*_{12,3} \check r_{12},\  
\quad  \check r_{23} F_{1,23} = F_{1,23}\check r_{23}. \label{comut1} 
\end{equation}
The detailed proof of (\ref{comut1})  is given in \cite{Doikoutw}.

This is a straightforward, but useful comment.  Recall, 
$r = {\cal P} \check r$, where ${\cal P}$ is the permutation operator.
If $F^*_{12,3} \check r_{12} = \check r_{12}F^*_{12,3},$ and 
$F_{1,23} \check r_{23} = \check r_{23}F_{1,23},$ 
then by multiplying the latter two equalities with ${\cal P}$ 
from the left we conclude: 
$F^*_{21,3}  r_{12} = r_{12}F^*_{12,3},$  and
$F_{1,32} r_{23} =  r_{23}F_{1,23}.$

A useful Corollary follows which is key in formulating our main Conjecture later.

\begin{cor}{\label{rem5}} Consider $F_{1,23}$ and $F^*_{12,3}$ defined in (\ref{cop1}) and (\ref{cop2a}) respectively. 
Then $F_{1,23} = (\mbox{id} \otimes \Delta)F,$ but $F^*_{12,3} \neq  (\Delta \otimes \mbox{id}) F.$ Also,  
if $\epsilon({\mathbb V}_{\eta}) =1,$ $\forall \eta \in X,$ then
$(\epsilon \otimes \mbox{id}) F  \neq I$ and $ (\mbox{id} \otimes \epsilon) F = I.$
\end{cor}
\begin{proof}
We recall that $F=\sum_{\eta\in X} e_{\eta, \eta} \otimes {\mathbb V}_{\eta}$, then $(\mbox{id} \otimes \Delta) F = 
\sum_{\eta \in X} e_{\eta, \eta} \otimes \Delta({\mathbb V}_{\eta}).$ Recalling also the form of $\Delta({\mathbb V}_{\eta})$ 
in (\ref{twistdelta}) and the defintion (\ref{cop1}) we conclude that $F_{1,23}=(\mbox{id} \otimes \Delta) F.$ In a similar fashion, 
$(\Delta \otimes \mbox{id}) F = \sum_{\eta \in X} \Delta(e_{\eta, \eta})\otimes \sum_{x\in X}e_{\sigma_{\eta}(x), x} \neq F^*_{12,3}$ (\ref{cop2a}).

Also, we can immediately deduce that $(\mbox{id} \otimes \epsilon) F = \sum_{\eta \in X} e_{\eta, \eta}=I$ and\\
$(\epsilon \otimes \mbox{id}) F  = \sum_{\eta \in X} \epsilon(e_{\eta, \eta}) {\mathbb V}_{\eta} \neq I,$ which is
 compatible with  Proposition \ref{twist}.
\end{proof}

The admissibility of the twist is proven in Proposition 3.15 in \cite{Doikoutw}, and is stated below:
\begin{pro}{\label{cocycle} (\cite{Doikoutw})}
 Let $F_{12} =F  \otimes I$ and $F_{23} =I \otimes F$, where 
$F=\sum_{\eta, x,y \in X} e_{\eta, \eta} \otimes e_{\sigma_{\eta}(x), x}$.
Let also  $F^*_{12,3}$ and $F_{1,23}$ defined in (\ref{cop1}) and (\ref{cop2a}). Then 
\begin{equation}
F_{123}:=F_{12} F^*_{12,3} =F_{23}F_{1,23}.
\end{equation}
\end{pro}
\noindent The proof is straightforward although tedious.
By substituting the expressions for $F_{12},\ F_{23}$, $F^*_{12,3}$ and $F_{1,23}$ 
(recall $C_1=0$ holds for $F^*_{12,3},$ $F_{1,23}$) we obtain by direct computation:
$F_{12}F ^*_{12,3} = F_{23}F_{1,23}.$
The explicit form the 3-twist is given from the expressions above as\\
$F_{123}= \sum_{\eta,x,y \in X}  
e_{\sigma_{\eta}(x), \sigma_{\eta}(x)} \otimes e_{\eta, \tau_x(\eta)} \otimes  
e_{\sigma_{\eta}(\sigma_{x}(y)),y}\vert_{C_1=0},$ \cite{Doikoutw}.

Given the findings of the first section based on Propositions \ref{lemma1}, \ref{twist}, Remark \ref{remdr}
and  Corollary \ref{rem5} we formulate the following conjecture:
\begin{conj}{\label{conj1}}
The element $F^*_{12,3}$ can be expressed as $F_{12,3}^* = \big( (\Delta \otimes \mbox{id}) F\big )\Phi^{-1},$ 
where $\Phi\in \mbox{End}(({\mathbb C}^n)^{\otimes 3})$ is an invertible element such that $\Phi_{213} r_{12} = r_{12} \Phi_{123}.$
\end{conj} 
\noindent This conjecture suggests that the quantum group emerging from set-theoretic solutions of the YBE is a quasi-bialgebra and is compatible with Propositions \ref{lemma1} and \ref{twist}.
A few comments are in order at this point: note that the explicit expression for $F^*_{12,3}$ is given in (\ref{cop2a}), then following the statement in Conjecture \ref{conj1} 
and recalling that $\Delta$ is an algebra homomorphism we can formally write $\Phi^{-1} = \big ((\Delta \otimes \mbox{id}) F^{-1}\big ) F_{12,3}^*.$ Also, it was shown in \cite{Doikoutw} that 
$\big [ F_{12,3}^*, \check r_{12} \big ] =0$ and by definition $\big [(\Delta \otimes \mbox{id}) F,\ \check r_{12} \big ] =0,$ so we conclude that $\Phi_{123}$ should also commute with $\check r_{12}$.

$ $

We will be discussing below the symmetries of the set-theoretic $r$-matrix and the corresponding Baxterized solutions.

\begin{cor}\label{prop2b} Let  $R_{Y}: {\mathbb C}^{n} \otimes {\mathbb C}^{n} \to 
{\mathbb C}^{n} \otimes {\mathbb C}^{n}$ be the Yangian 
$R-$matrix and $R(\lambda) = \lambda r +{\cal P},$
where $r$ is the set-theoretic solution $r = \sum_{x,y \in X} e_{y,\sigma_{x}(y)} \otimes e_{x, \tau_y(x)}.$ Then 
\begin{equation}
\Delta^{(op)}(e_{z,w}) R(\delta) = R(\delta) \Delta(e_{z, w}), ~~\Delta^{(op)}(f_{z,w}; \lambda_1, \lambda_2) R(\lambda) = 
R(\lambda) \Delta(f_{z,w}; \lambda_1, \lambda_2)
\end{equation}
where $\delta : = \lambda_1-  \lambda_2, $ and
\begin{eqnarray}
&& \Delta(e_{z,w}) = \sum_{\xi, \zeta \in X} \big (e_{z,w} \otimes e_{\xi, \zeta} + e_{\sigma_z(\xi), \sigma_{w}(\zeta)} 
\otimes e_{\tau_{\xi}(z), \tau_{\zeta(w)}}\big )_{\sigma_z(\xi)=\sigma_w(\zeta)}\cr
&& \Delta(f_{z,w}; \lambda_1, \lambda_2)=  \sum_{\xi, \zeta \in X} \big (\lambda_1 e_{z,w} \otimes
 e_{\xi, \zeta} + \lambda_2 e_{\sigma_z(\xi), \sigma_{w}(\zeta)} \otimes e_{\tau_{\xi}(z), \tau_{\zeta(w)}}\big )_{\sigma_z(\xi)=\sigma_w(\zeta)}. \nonumber \\
&& \qquad \qquad \qquad \quad {1\over 2} \sum_{y, \hat y \in X}\big(e_{z,\sigma_z(y)}
\otimes e_{y,\hat y}\vert_{w = \sigma_{\sigma_z(y)}(\hat y)} -e_{\sigma_{w}(y), w} \otimes e_{\hat y,y}\vert_{z = \sigma_{\sigma_w(y)}(\hat y)} \big )
\label{twistdeltab}
\end{eqnarray}
Moreover, the matrix $\check R(\lambda) = {\cal P} R(\lambda)$ is $\mathfrak{gl}_n$ symmetric.
\end{cor}
\begin{proof}
The Yangian $R$-matrix satisfies relations (\ref{basicY}), (\ref{coY}), then 
$R(\delta)\Delta(e_{z,w}) = \Delta^{(op)}(e_{z,w})R(\delta)$ and 
$R(\delta)\Delta(f_{z,w}; \lambda_1, \lambda_2) = \Delta^{(op)}(f_{z,w}; \lambda_1, \lambda_2)R(\delta),$ 
where $R(\lambda) = \lambda r +{\cal P}$ and $r$ is the set-theoretic solution 
$r = \sum_{x,y \in X} e_{y \sigma_{x}(y)} \otimes e_{x, \tau_y(x)}$, also
\begin{eqnarray} 
&& \Delta(e_{z,w}) =F^{-1}\big ( e_{z,w}  \otimes I + I \otimes e_{z,w}\big) F, \nonumber\\ 
&& \Delta(f_{z,w}, \lambda_1, \lambda_2) = F^{-1}\big (\lambda_1 e_{z,w} \otimes I + \lambda_2 I \otimes   
e_{z,w} + {1\over 2}\sum_{ y \in X} (e_{z,y} \otimes e_{y,w} - e_{y,w} \otimes e_{z,y} )\big) F, \label{sim1}
\end{eqnarray}
where we recall $F = \sum_{\eta, x \in X} e_{x,x} \otimes {\mathbb V}_x.$  Explicit computation of the latter leads to (\ref{twistdeltab}).

Also, $\Delta$ is an algebra homomorphism,  hence $\Delta(e_{x, y})$ also satisfy the $\mathfrak{gl}_n$ algebra relations, 
thus $\check R = {\cal P} R$  is $\mathfrak{gl}_n$ symmetric ie. $\big [ \check R(\lambda),\ \Delta(e_{x,y}) \big ] =0 $. 
\end{proof}

\subsection*{The Lyubashenko solution} We finish this section with a discussion of 
a special set-theoretic solution known 
as Lyubashenko's solution that supports Conjecture \ref{conj1}.
The analysis for this special class of set-theoretic solutions is compatible with 
Remark \ref{specialcase} and Example \ref{cor1L}.

The Lyubashenko solution is defined as
\begin{equation}
\check r=\sum_{x, y \in X} e_{x, \sigma(y)} \otimes e_{y, \tau(x)}, \label{special1}
\end{equation}
where $\tau,\ \sigma: X \to X$ are bijective functions, 
such that $\sigma(\tau(x)) = \tau(\sigma(x)) = x.$ Let ${\mathbb V} = \sum_{x\in X} e_{x, \tau(x)},$
then as was shown in \cite{DoiSmo2} the special solution 
(\ref{special1}) can be obtained from 
the permutation operator as 
$\check r =  ( I \otimes {\mathbb V}^{-1} ) {\cal P} (I \otimes {\mathbb V}),$ which leads to (\ref{special1}).
Also, $r = {\cal P} \check r$  and  the Baxterized solution $R(\lambda)$ 
have a simple form for this class of solutions:
\begin{equation}
r= {\mathbb V}^{-1} \otimes {\mathbb V}\ \Rightarrow\  R(\lambda) = \lambda {\mathbb V}^{-1} \otimes {\mathbb V} + {\cal P}. \label{special2}
\end{equation}

\noindent Some examples of the above construction are given below:

1. $\sigma(x) =x+1,\ \tau(x) =x-1$,  where addition is modulo $n$ (see also \cite{LAA}).

2. $\sigma(x) =n +1-x,\ \tau(x) =n+1-x$.

In both examples $x \in \{ 1, \ldots, n\}.$\\

In the case of Lyubashenko solution, we consider Example \ref{cor1L} (special case), where  we set ${\cal A}$ to be the Yangian ${\cal Y}(\mathfrak{gl}_n)$, i.e. 
$\big( {\cal Y}(\mathfrak{gl}_n), \Delta, \epsilon, \Phi, {\cal R} \big )$ ($\Phi = 1 \otimes 1 \otimes u^{-1}$) is a quasi-triangular quasi-bialgebra. 
We recall the evaluation representation 
$\pi_{\lambda}: {\cal Y}(\mathfrak{gl}_n) \to \mbox{End}({\mathbb C}^n),$ such that $ u \mapsto {\mathbb V},$  
${\mathrm Q}_{xy}^{(1)} \mapsto e_{x,y},$  ${\mathrm Q}_{xy}^{(2)} \mapsto \lambda e_{x,y}$ etc.
Then the twist becomes $F = I \otimes {\mathbb V}$ and the twisted coproducts are
\begin{eqnarray} 
&& \Delta(e_{z,w}) =  e_{z,w} \otimes  I + I \otimes   e_{\tau(z) ,\tau(w)},  \nonumber \\
&&   \Delta(f_{x,y}; \lambda_1, \lambda_2) =  \lambda_1e_{z,w} \otimes  I +\lambda_2 I \otimes   e_{\tau(z) ,\tau(w)} \ + {1\over 2}\sum_{y \in X} 
\big (e_{z, \sigma(y)} \otimes e_{y,\tau(w)} - e_{\sigma(y), w} \otimes
e_{\tau(z), y} \big ).  \nonumber
\end{eqnarray}

Recall that in this case $F= I \otimes {\mathbb V}$, and $\Phi = I \otimes I \otimes {\mathbb V}^{-1},$ 
then we have:
\begin{eqnarray}
&& F_{1,23} := (\mbox{id} \otimes \Delta)F = I \otimes {\mathbb V} \otimes {\mathbb V} \nonumber \\
&& F_{12,3}^* := \big( (\Delta \otimes \mbox{id})F \big )\Phi^{-1}= I \otimes I \otimes {\mathbb V}^2,\nonumber
\end{eqnarray}
in accordance to the Conjecture \ref{conj1}.
Also, the $N$-fold twist is $ F_{12...N} = I \otimes {\mathbb V} \otimes {\mathbb V}^{2}\otimes \ldots \otimes {\mathbb V}^{(N-1)}$  (see also relevant findings in \cite{Doikoutw}).

\section{Quasi-bialgebras from ${\mathfrak U}_q(\widehat{\mathfrak{gl}_n})$}

\noindent We will discuss in this section the $q$-generalizations of  set-theoretic solutions.  
Although the situations we are going to address here strictly speaking are not
set-theoretic solutions of the Yang-Baxter equation 
they are certainly inspired by the results of the preceding section.  After a brief review on ${\mathfrak U}_q(\widehat{\mathfrak{gl}_n})$ we will consider the $q$-deformed analogues of the set-theoretic solutions via the twists discussed in the previous section and in \cite{Doikoutw}, 
subject to certain extra constraints. The findings of this section greatly generalize the preliminary results of \cite{DoiSmo2}.

\subsection{The algebra ${\mathfrak U}_q(\widehat{\mathfrak{gl}_n})$}

\noindent It will be useful in what follows to recall the basic definitions regarding the algebra
${\mathfrak U}_q(\widehat{\mathfrak{gl}_n})$ \cite{Drinfeld, Jimbo, Jimbo2}.
Let \be a_{ij} = 2 \delta_{ij} - (\delta_{i\ j+1}+ \delta_{i\ j-1} +\delta_{i1}\ \delta_{jn}+\delta_{in}\ \delta_{j1}), ~~i,\ j\in \{1, \ldots , n \} \ee 
 be the Cartan matrix of the affine Lie algebra
${\widehat{\mathfrak{sl}_n}}$\footnote{For the $\widehat{\mathfrak{sl}_{2}}$ case in particular \be a_{ij} =2\delta_{ij} -2 (\delta_{i1}\ \delta_{j2} 
+\delta_{i2}\ \delta_{j1}), ~~i,\ j \in \{ 1, 2\}\ee}. Also define: 
\begin{eqnarray} && [m]_{q} ={q^{m} -q^{-m} 
\over q -q^{-1}}, ~~~[m]_{q}!= \prod_{k=1}^{m}\ [k]_{q},~~~[0]_{q}! =1 \non\\  
&&\left [ \begin{array}{c}
m \\
n \\ \end{array} \right  ]_{q} = {[m]_{q}! \over [n]_{q}!\ [m-n]_{q}!}, ~~~m>n>0. 
 \end{eqnarray}
\begin{defn}\label{defq} {\it The quantum affine enveloping
algebra ${\mathfrak U}_q(\widehat{\mathfrak{sl}_n})$ has the
Chevalley-Serre generators}  $e_{i}$, $f_{i}$,
$q^{\pm {h_{i}\over 2}}$, $i\in \{1, \ldots, n\}$ {\it obeying the defining relations:} 
\begin{eqnarray}
&& \Big [q^{\pm {h_{i}\over 2}},\ q^{\pm {h_{j}\over 2}} \Big]=0\, \qquad q^{{h_{i}\over 2}}\ e_{j}=q^{{1\over
2}a_{ij}}e_{j}\ q^{{h_{i}\over 2}}\, \qquad q^{{h_{i}\over 2}}\ f_{j}
= q^{-{1\over 2}a_{ij}}f_{j}\ q^{{h_{i}\over 2}}, \non\\
&& \Big [e_{i},\ f_{j}\Big ] = \delta_{ij}{q^{h_{i}}-q^{-h_{i}} \over q-q^{-1}},
~~~~i,j \in \{ 1, \ldots,n \}
\label{1} 
\end{eqnarray}
{\it and the $q$ deformed Serre relations 
\begin{eqnarray} 
&& \sum_{n=0}^{1-a_{ij}} (-1)^{n}
\left [ \begin{array}{c}
  1-a_{ij} \\
   n \\ \end{array} \right  ]_{q} 
\chi_{i}^{1-a_{ij}-n}\ \chi_{j}\ \chi_{i}^{n} =0, ~~~\chi_{i} \in \{e_{i},\ f_{i} \}, ~~~ i \neq j. \label{chev} 
\end{eqnarray}}
\end{defn}

\begin{rem} \label{rem3q}   The generators $e_{i}$, $f_{i}$, $q^{\pm h_{i}}$ for $i\in \{1, \ldots, n\}$ form the 
${\mathfrak U}_q(\widehat{\mathfrak{sl}_n})$ algebra. 
Also, $q^{\pm h_{i}}=q^{\pm (\varepsilon_{i} -\varepsilon_{i+1})}$,  $i\in \{1, \ldots, n-1\}$ and 
$q^{\pm h_{n}}=q^{\pm (\varepsilon_{n} -\varepsilon_{1})},$ where the elements
$q^{\pm \varepsilon_{i}}$ belong to ${\mathfrak U}_q(\widehat{{\mathfrak{gl}}_n}) $. Recall that ${\mathfrak U}_q(\widehat{{\mathfrak{gl}_n}}) $ is obtained by adding to 
${\mathfrak U}_q(\widehat{{\mathfrak{sl}_n}}) $ the elements $q^{\pm \varepsilon_{i}}$ $i\in \{1, \ldots, n\}$  so that $q^{\sum_{i=1}^{n}\varepsilon_{i}}$ 
belongs to the center (for more details see \cite{Jimbo}).

\end{rem}

We also note that $\big ( {\mathfrak U}_q(\widehat{\mathfrak{gl}_n}) , \Delta_q, \epsilon, S_q, {\cal R}_{q}\big )$ 
is a quasi-triangular Hopf algebra over ${\mathbb C}$  equipped with \cite{Drinfeld}:
\begin{itemize}
\item A coproduct $\Delta_q:\  {\mathfrak U}_q(\widehat{\mathfrak{gl}_n})
\to {\mathfrak U}_q(\widehat{\mathfrak{gl}_n}) \otimes {\mathfrak U}_q(\widehat{\mathfrak{gl}_n})$ such that
\begin{eqnarray}
 && \Delta_q(\xi_i) = q^{- {h_{i} \over 2}} \otimes \xi_i + \xi _i\otimes q^{{h_{i} \over 2}}, ~~\xi_i \in \Big \{e_{i},\ f_{i}\Big  \} \label{copba} \\
&&\Delta_q(q^{\pm{\varepsilon_{i} \over 2}}) = q^{\pm{\varepsilon_{i} \over 2}} \otimes
q^{\pm{\varepsilon_{i} \over 2}}, \quad i \in \{1, \ldots,n \} .\label{copbb} 
\end{eqnarray} 
The $l$-fold co-product 
$\Delta_q^{(l)}:\ {\mathfrak U}_q(\widehat{\mathfrak{gl}_n})
\to {\mathfrak U}_q(\widehat{\mathfrak{gl}_n})^{\otimes (l)}$ is defined as 
$\Delta_q^{(l)} = (\mbox{id} \otimes \Delta_q^{(l-1)})\Delta_q =(\Delta_q^{(l-1)} \otimes \mbox{id})
\Delta_q.$

\item A co-unit $\epsilon: {\mathfrak U}_q(\widehat{\mathfrak{gl}_n}) \to {\mathbb C}$ such that
\begin{equation}
\epsilon(e_j) = \epsilon(f_j) =0, ~~~~\epsilon(q^{\varepsilon_j}) =1.
\end{equation}

\item An antipode $S_q:  {\mathfrak U}_q(\widehat{\mathfrak{gl}_n}) \to  {\mathfrak U}_q(\widehat{\mathfrak{gl}_n})$ such that
\begin{equation}
S_q(q^{\varepsilon_i}) = q^{-\varepsilon_i} , ~~~S_q(\xi_i) =- q^{h_i\over 2} \xi_i q^{-h_i\over 2}.
\end{equation}

\item There exists an invertible element ${\cal R}_q \in {\mathfrak U}_q(\widehat{\mathfrak{gl}_n}) \otimes {\mathfrak U}_q(\widehat{\mathfrak{gl}_n}),$ such that it satisfies the axioms of Definition \ref{definition2} for 
$\Phi =1\otimes 1 \otimes 1.$

\end{itemize}

We recall now Example \ref{cor1L} and we briefly discuss the notion of the antipode for ${\mathfrak U}_q(\widehat{\mathfrak{gl}_n}).$  We fist introduce a shorthand notation for the unit element of ${\mathfrak U}_q(\widehat{\mathfrak{gl}_n})$: $1_q := 1_{{\mathfrak U}_q(\widehat{\mathfrak{gl}_n})}.$

\begin{rem}{\label{remanti2}}
We specialize Example \ref{cor1L} to the case of ${\mathfrak U}_q(\widehat{\mathfrak{gl}_n}),$  i.e. 
$\big({\mathfrak U}_q(\widehat{\mathfrak{gl}_n}), \Delta, \epsilon, \Phi, {\cal R} \big ),$ where $\Phi = 1 \otimes 1 \otimes u^{-1}$ for some group-like element $u$,
is a quasi-triangular quasi-bialgebra.  We now want to test whether 
$\big ({\mathfrak U}_q(\widehat{\mathfrak{gl}_n}),  \Delta, \epsilon,\Phi,  {\cal R}\big)$ is a quasi-Hopf algebra,  
i.e.  if there exist $S$, $\alpha$, $\beta$ such that axioms (1) and (2) in Definition \ref{definition1b} hold.
Suppose that the group-like element $u$ exists,  then $\Delta(u) = u \otimes u$ and:
\begin{enumerate}
\item $S(u)\alpha u =\alpha$ and $u\beta S(u) =\beta,$ which lead to $S(u) = \alpha u^{-1} \alpha^{-1} = \beta^{-1} u^{-1}\beta.$
\item $\beta \alpha u^{-1} =1_{q}$ and $\alpha \beta S(u) = 1_q,$ which lead to $u = \beta \alpha$ and $S(u) = \beta^{-1} \alpha^{-1}.$
\end{enumerate} 
All the above equations are self-consistent. Similarly, we can show that  $S(u^{-1}) = \alpha \beta$

We now check the axioms of Definition \ref{definition1b} for the elements $e_j,\ f_j,\ q^{{\varepsilon_j}},$  with coproducts, after twisting, given as
$\Delta(\xi_j) = q^{-{h_j \over 2}} \otimes u^{-1}\xi_ju + \xi_j\otimes u^{-1} q^{{h_j \over 2}} u,$ $~\xi_j\in \{e_j,\ f_j\}$ and 
$\Delta(q^{\varepsilon_j}) = q^{\varepsilon_j} \otimes u^{-1} q^{\varepsilon_j} u$:
\begin{enumerate}
\item We consider $S(q^{\varepsilon_j}) \alpha u^{-1}q^{\varepsilon_j}u=\alpha$ and  $q^{\varepsilon_j} \beta S(u^{-1}q^{\varepsilon_j}u) = \beta,$ 
which lead to (by requiring that $S$ is an algebra anti-homomorphism): $ S(q^{\varepsilon_j}) =
\alpha q^{-\varepsilon_j}\alpha^{-1}  = \beta^{-1}q^{-\varepsilon_j}\beta.$
The latter expression leads to $\big [ q^{\varepsilon_j},\ u\big ]=0.$ Similar conclusions hold  for $S(\xi_j).$ Indeed, from axiom (1) of Definition \ref{definition1b} we conclude that $S(\xi_j) = -\alpha q^{h_j \over 2}\xi_j \ q^{-{h_j \over 2}}\alpha^{-1}=-\beta^{-1} q^{h_j \over 2}\xi_j \ q^{-{h_j \over 2}}\beta$, which leads to $\big [ \xi_j,\ u\big] = 0.$
\end{enumerate}

\noindent As in the case of the Yangian studied in the previous section, the axioms for the antipode restrict $u$ to be in the center of the algebra, which in general is not true. Indeed,  simple examples of such group-like elements that are not central are given by $u = q^{\pm \varepsilon_j}$. Hence, we can not consistently define $S(w), \ w \in {\mathfrak U}_q(\widehat{\mathfrak{gl}_n})$ for a generic $u$
by strictly following the axioms of Definition \ref{definition1b}.
\end{rem}

It will be useful for the findings of the next subsection to recall the evaluation representation of 
${\mathfrak U}_q\widehat{(\mathfrak{gl}_n})$ \cite{Jimbo} (homogeneous gradation),
$\pi_{\lambda}: {\mathfrak U}_q(\widehat{\mathfrak{gl}_n})\to \mbox{End}({\mathbb C}^n),$ $\lambda \in {\mathbb C}$:
\begin{eqnarray} 
&&\pi_{\lambda}(e_{i})= e_{i, i+1}, ~~~\pi_{\lambda}(f_{i})=e_{i+1, i}, ~~~\pi_{\lambda}(q^{{\varepsilon_{i} \over 2}}) = 
q^{{ e_{i,i} \over 2}},~~ i \in \{1,\ldots, n-1\} \cr
&& \pi_{\lambda}(e_n) = e^{-2\lambda} e_{n,1},  ~~~ \pi_{\lambda}(f_n) = e^{2\lambda} e_{1,n},  ~~~ ~
\pi_{\lambda}(q^{{h_{n} \over 2}}) = e^{{e_{n,n} - e_{1,1} \over 2}}\label{eval} 
\end{eqnarray}
and we also introduce some useful notation:
\begin{eqnarray}
&& (\pi_{\lambda_1} \otimes \pi_{\lambda_2})\Delta_q(e_j) = \Delta_q(e_{j, j+1}),  ~~ (\pi_{\lambda_1} 
\otimes \pi_{\lambda_2})\Delta_q(f_j) = \Delta_q(e_{ j+1, j}), \cr && j\in \{1,\ldots, n-1\},  \cr
&&  (\pi_{\lambda_1} \otimes \pi_{\lambda_2})\Delta_q(e_n) = \Delta_q(e_{n,1}; \lambda_1, \lambda_2),  
~~ (\pi_{\lambda_1} \otimes \pi_{\lambda_2})\Delta_q(f_n) = \Delta_q(e_{ 1,n};\lambda_1, \lambda_2)\cr
&& (\pi_{\lambda_1} \otimes \pi_{\lambda_1})\Delta_q(q^{\epsilon_j}) = \Delta_q(q^{e_{j, j}}), ~ j\in \{1,\ldots, n\}. \label{notation2q}
\end{eqnarray}

We recall  also the ${\mathfrak U}_q(\mathfrak{gl}_n)$-invariant representation of the $A$-type Hecke algebra \cite{Jimbo2}:
\begin{equation}
{\mathrm g} =\sum_{x \neq y\in X} \Big ( e_{x, y} \otimes e_{y,x}  - q^{-sgn(x-y)} e_{x,x} \otimes e_{y,y} \Big ) +q I. \label{braidq}
\end{equation}
Indeed, the above element  satisfies the braid relation $(I\otimes g)(g\otimes I)(I\otimes g) =(g\otimes I)(I\otimes g) (g\otimes I) $ 
as well as the Hecke constraint $(g -q)(g + q^{-1}) =0.$
The Baxterized ${\mathfrak U}_q(\widehat{\mathfrak{gl}_n})$ solution of the Yang Baxter equation is the
$R_q(\lambda) = e^{\lambda} {\mathrm g}^+ - e^{-\lambda} {\mathrm g}^{-}, $ where ${\mathrm g}^+ = {\cal P} {\mathrm g},$
${\mathrm g}^- = {\cal P} {\mathrm g}^{-1}$ (${\cal P}$ the permutation operator).  We also define 
$\Delta_q^{(op)}(w;\lambda_1, \lambda_2) := {\cal P} \Delta_q(w;\lambda_2, \lambda_1) {\cal P},$ the 
${\mathfrak U}_q(\widehat{\mathfrak{gl}_n})$ 
$R_q$-matrix satisfies the intertwining relations:
\begin{eqnarray}
&& R_q(\lambda_1-\lambda_2)\Delta_q(\zeta) = \Delta_q^{(op)}(\zeta) R_q(\lambda_1 -\lambda_2)~~~~\zeta \in 
\{ e_{j,j+1}, e_{j+1,j}, q^{e_{j,j}}, q^{e_{n,n}}\},\cr && j\in \{1, \ldots, n-1 \} \cr
&& R_q(\lambda_1-\lambda_2)\Delta_q(\zeta_n;\lambda_1, \lambda_2) = \Delta_q^{(op)}(\zeta_n; \lambda_1, \lambda_2)
 R_q(\lambda_1 -\lambda_2), ~~~\zeta_n \in \{e_{n,1}, e_{1,n,}\}.\label{interbq}
\end{eqnarray}

This brief review on ${\mathfrak U}_q(\widehat{\mathfrak{gl}_n})$  will be particularly 
useful for the findings of the subsequent subsection.

\subsection{The $q$-analogues of set-theoretic solutions of the YBE $\&$ quasi-bialgebras}

\noindent  Inspired by the set-theoretic solutions and the associated twists \cite{Doikoutw},  as discussed in the previous section, we generalize 
in what follows results regarding the twist of the ${\mathfrak U}(\widehat{\mathfrak{gl}_n})$ $R$-matrix.
Note that strictly speaking this solution is not a set-theoretic solution of the braid equation. Nevertheless, 
the admissible twists found for the set-theoretic solutions can be still exploited to yield generalized solutions
based on (\ref{braidq}).

We state below a basic Lemma that will lead to the main Proposition of this section associated to admissible twists of 
the ${\mathfrak U}_q(\widehat{\mathfrak{gl}_n})$ $R$-matrix. This construction provides  the $q$-analogue of the 
$R$-matrices coming from set-theoretic solutions of the YBE and greatly generalizes the preliminary results of \cite{DoiSmo2}.
\begin{lemma}{\label{lemq1}}
Let ${\mathbb V}_{\eta} = \sum_{x\in X}e_{\sigma_{\eta}(x), x},$ $\forall \eta \in X$ consist a family of group-like elements, i.e. 
$\Delta_{q}({\mathbb V}_{\eta}) = {\mathbb V}_{\eta} \otimes {\mathbb V}_{\eta},$  and let ${\mathrm g}$ be the ${\mathfrak U}({\mathfrak{gl}_n})$-invariant 
element (\ref{braidq}). Then $\big [{\mathrm g},\ \Delta_q({\mathbb V}_{\eta}) \big ] =0,$ subject to the constraint 
$sgn (x-y) = sgn \big (\sigma_{\eta} (x) - \sigma_{\eta}(y)\big ),$ $\forall \eta, x, y \in X.$
\end{lemma}
\begin{proof}
It is convenient to re-express the element ${\mathrm g}$ as
\begin{equation}
{\mathrm g} = {\cal P} - \sum_{x, y \in X} q^{-sgn(x-y)} e_{x,x} \otimes e_{y,y} + q I, \label{rebraid}
\end{equation}
where ${\cal P}$ is the permutation operator and $I$ is the $n \times n$ identity matrix. It is obvious that 
$\Delta_q({\mathbb V}_{\eta})$ 
commutes with ${\cal P}$, so it suffices to show that 
$\Delta_q({\mathbb V}_{\eta})$ commutes with the second term of (\ref{rebraid}), indeed we compute
\begin{equation}
\big (\sum_{x, y \in X} q^{-sgn(x-y)} e_{x,x} \otimes e_{y,y}\big ) \big ( \sum_{\xi,\zeta \in X} e_{\sigma_{\eta}(\xi), \xi} \otimes e_{\sigma_{\eta}(\zeta), \zeta}\big ) =
 \sum_{\zeta, \xi \in X} q^{-sgn \big(\sigma_{\eta}(\xi)-\sigma_{\eta}(\zeta) \big )}  e_{\sigma_{\eta}(\xi), \xi} \otimes  e_{\sigma_{\eta}(\zeta), \zeta} \label{expa}
\end{equation}
\begin{equation}
 \big ( \sum_{\xi,\zeta \in X} e_{\sigma_{\eta}(\xi), \xi} \otimes e_{\sigma_{\eta}(\zeta), \zeta}\big )\big 
(\sum_{x, y \in X} q^{-sgn(x-y )} e_{x,x} \otimes e_{y,y}\big ) =\sum_{\zeta, \xi \in X} q^{-sgn (\xi-\zeta )}  
e_{\sigma_{\eta}(\xi), \xi} \otimes  e_{\sigma_{\eta}(\zeta), \zeta}.  \label{expb}
\end{equation}
Requiring expressions (\ref{expa}) and (\ref{expb}) to be equal we conclude that $[g,\Delta_q(V_\eta)]=0, $ 
$\forall \eta \in X$, if and only if $\sigma_\eta$ is an order-preserving permutation of $X$, i.e.  $sgn(x-y) = 
sgn \big (\sigma_{\eta} (x) - \sigma_{\eta}(y)\big ),$ $\forall \eta, x, y \in X.$ 
\end{proof}

If $X$ is finite then the identity map is the only order-preserving permutation.
It is important however to note that the condition $sgn(x-y) = 
sgn \big (\sigma_{\eta} (x) - \sigma_{\eta}(y)\big )$ possibly holds when considering countably infinite sets.  A characteristic example is demonstrated at the end of the section, where the $q$-analogue of the Lyubaschenko solution is presented.  A further detailed analysis of this construction and possible generalizations will be 
presented in future works.

We come now to the main proposition of this section.
\begin{pro} Let ${\mathrm g}$ be the ${\mathfrak U}({\mathfrak{gl}_n})$-invariant 
element (\ref{braidq}) and  $ F = \sum_{\eta \in X} e_{\eta, \eta} \otimes {\mathbb V}_{\eta}$ (${\mathbb V}_{\eta} = \sum_{x\in X}e_{\sigma_{\eta}(x), x}$)
be the set-theoretic twist. Let also $F_{1,23}$ and $F_{12,3}^*$ be the quantities defined in (\ref{cop1}), (\ref{cop2a}). 
We also define ${\mathrm G} = F^{-1}{\mathrm g} F,$ then provided that $C_1 =0,$ $C=0$ 
($C_1 = \sigma_{\eta}(\sigma_x(y)) -\sigma_{\sigma_{\eta}(x)}(\sigma_{\tau_x(\eta)}(y)),$ 
$C= sgn(x-y) - sgn\big (\sigma_{\eta}(x) - \sigma_{\eta}(y)\big )$):
\begin{enumerate}
\item $\big [ {\mathrm G}_{12}, F_{12,3}^* \big ] =\big [ {\mathrm G}_{23}, F_{1,23} \big ] = 0.$
\item The twisted element ${\mathrm G}$  satisfies the braid relation and the Hecke constraint (i.e.  it  provides a representation of the $A$-type Hecke algebra).
\end{enumerate}
\end{pro}
\begin{proof}
We showed in Lemma \ref{lemq1} that $\big [{\mathrm g},\ \Delta_q({\mathbb V}_{\eta}) \big ] =0,$ subject to the constraint $C=0.$ This leads to
$\big [{\mathrm G},\ \Delta({\mathbb V}_{\eta}) \big ] =0,$ $\forall \eta \in X,$ where ${\mathrm G} = F^{-1}{\mathrm g} F$ 
and $\Delta({\mathbb V}_{\eta}) = F^{-1} \Delta_q({\mathbb V}_{\eta})F,$ where $\Delta_q({\mathbb V}_{\eta})$ is derived subject to $C_1=0$ (see Lemma \ref{lemq1}).

\begin{enumerate}
\item We recall that $F_{1,23} = (\mbox{id} \otimes \Delta) F,$ i.e. $F_{1,23} =\sum_{\eta \in X} e_{\eta, \eta} \otimes \Delta({\mathbb V}_{\eta}),$ 
then using $\big [{\mathrm G},\ \Delta({\mathbb V}_{\eta})\big] =0$ we obtain, $\big [ {\mathrm G}_{23}, F_{1,23} \big ] = 0,$ subject to $C_1 =C=0.$ 

We show now that $\big [ {\mathrm G}_{12}, F_{12,3}^* \big ] =0,$ to achieve this we explicitly derive the form of ${\mathrm G}:$
\begin{eqnarray}
&& {\mathrm G} = F^{-1} {\mathrm g} F = \check r - \sum_{x,y\in X} q^{-sgn\big (x - \sigma_x(y)\big )} e_{x,x} \otimes e_{y,y} + qI,
\end{eqnarray}
where $\check r = \sum_{x,y \in X} e_{x, \sigma_x(y)} \otimes e_{y, \tau_y(x)}.$
We recall also the explicit expression $F^*_{12,3} = \sum_{\eta, x, y\in X} e_{\eta, \eta} \otimes e_{x, x} 
\otimes e_{\sigma_{\eta}(\sigma_x(y))}|_{C_1=0}$ and we  compute:
\begin{eqnarray}
&& ~{\mathrm G}_{12} F_{12,3}^* = F_{12,3}^*{\mathrm G}_{12}  =\\
&&  \sum_{x,y, z \in X} \big( e_{x, \sigma_x(y)} \otimes e_{y, \tau_y(x)}   - q^{-sgn\big (x- \sigma_{x}(y)\big)} e_{x,x} 
\otimes e_{y,y}\big ) \otimes e_{\sigma_x(\sigma_y(z)),z}  + qF^*_{12,3}. \nonumber
\end{eqnarray}
\item Due to the fact that: (1) $F$ is an admissible Drinfeld twist,  (2) $\big [ {\mathrm G}_{12}, F_{12,3}^* \big ] =
\big [ {\mathrm G}_{23}, F_{1,23} \big ] = 0,$ and (3) ${\mathrm g}$ satisfies the braid relation, we conclude that 
${\mathrm G}$ also satisfies the braid equation. Similarly,
 $({\mathrm g} - q) ({\mathrm g}+ q^{-1}) = 0 \Rightarrow ({\mathrm G} - q) ({\mathrm G}+ q^{-1}) = 0.$
\hfill \qedhere
\end{enumerate}
\end{proof}

\begin{rem}{\label{remq1}}
The Baxterized solution 
$\check R(\lambda ) = F^{-1} \check R_q(\lambda) F = e^{\lambda} {\mathrm G} - e^{-\lambda } {\mathrm G}^{-1}$  
(recall $\check R_q(\lambda)= e^{\lambda} {\mathrm g} - e^{-\lambda } {\mathrm g}^{-1}$), and hence 
$R(\lambda)= {\cal P} \check R(\lambda)  =  e^{\lambda} {\mathrm G}^+ - e^{-\lambda } {\mathrm G}^{-},$
which is a solution of the Yang-Baxter equation.
Moreover, by means of (\ref{interbq}) and the $R$-matrix twist we conclude that
\begin{eqnarray}
&& R(\lambda_1-\lambda_2)\Delta(\zeta) = \Delta^{(op)}(\zeta) R_q(\lambda_1 -\lambda_2)~~~~\zeta \in 
\{ e_{j,j+1}, e_{j+1,j}, q^{e_{j,j}},  q^{e_{n,n}}\}, \label{q11} \\ && j\in \{1, \ldots n-1 \}  \cr
&&  R(\lambda_1-\lambda_2)\Delta(\zeta_n;\lambda_1, \lambda_2) = \Delta^{(op)}(\zeta_n; \lambda_1, \lambda_2)
 R(\lambda_1 -\lambda_2), ~~~\zeta_n\in \{e_{n,1}, e_{1,n}\},\label{q12}
\end{eqnarray}
where $\Delta(a) = F^{-1} \Delta_q(a)F,$ $~\forall a \in {\mathfrak U}(\widehat{\mathfrak{gl}_n}).$ 
From (\ref{q11}) and recalling that $\check R = {\cal P} R$ we deduce
\begin{eqnarray}
&& \check R(\lambda_1-\lambda_2)\Delta(\zeta) = \Delta(\zeta) \check R(\lambda_1 -\lambda_2)~~~~\zeta \in 
\{ e_{j,j+1}, e_{j+1,j}, q^{e_{j,j}},  q^{e_{n,n}} \}, \label{q21} \\ && j\in \{1, \ldots, n-1 \}, \nonumber
\end{eqnarray}
i.e. the $\check R$-matrix is ${\mathfrak U}_q({\mathfrak{gl}_n})$-invariant as is the $\check R_q$-matrix.
\end{rem}

\subsection*{The $q$-Lyubashenko solution}
We focus in the end of this section on a special example, the $q$-deformed analogue of Lyubashenko's solution introduced in \cite{DoiSmo2}. 
The analysis for this special class of $q$ set-theoretic like solutions is 
compatible with Remark \ref{specialcase} and Example \ref{cor1L}.

Let $\tau, \sigma: X \to X$ be isomorprhisms, 
such that $\sigma(\tau(x)) = \tau(\sigma(x)) = x$ and let ${\mathbb V}= \sum_{x\in X}e_{x,\tau(x)},$
then  by direct computation it follows that, $\big [ {\mathbb V}\otimes {\mathbb V},\ {\mathrm g}\big]=0 $ 
provided that $sgn(x-y) = sgn\big (\sigma(x)- \sigma(y)\big )$. We then define ${\mathrm G} = (I\otimes {\mathbb V}^{-1})\ {\mathrm g}\  (I \otimes {\mathbb V}),$
which leads to
\begin{eqnarray}
{\mathrm G}=  \sum_{x, y \in X}\Big ( e_{x, \sigma(y)} \otimes e_{y, \tau(x)}  - q^{-sgn(x-\sigma(y))} e_{x,x} 
\otimes e_{y,y} \Big )+ qI\label{specialq}
\end{eqnarray}
i.e. the element ${\mathrm G}$ is obtained from  the ${\mathfrak U}_q(\mathfrak{gl}_n)$-invariant braid solution (\ref{braidq}), provided that $sgn(x-y) =  sgn\big (\sigma(x)- \sigma(y)\big )$ \cite{DoiSmo2}. 

An example compatible with the construction above is:  $\sigma(x) =x+1,\ \tau(x) =x-1$,  $x \in {\mathbb Z}$ 
(see also \cite{LAA}), \\

We recall  Example \ref{cor1L}, where  
we set ${\cal A}$ to be  ${\mathfrak U}_q(\widehat{\mathfrak{gl}_n})$, i.e. 
$\big( {\mathfrak U}_q(\widehat{\mathfrak{gl}_n}), \Delta, \epsilon, \Phi, {\cal R} \big )$ 
($\Phi = 1 \otimes 1 \otimes u^{-1}$) 
is a quasi-triangular quasi-bialgebra. 
We recall the evaluation representation 
$\pi_{\lambda}: {\mathfrak U}_q(\widehat{\mathfrak{gl}_n}) \to \mbox{End}({\mathbb C}^n),$ such that $ u \mapsto {\mathbb V},$  
$e_j \mapsto e_{i,j+1},$  $f_j \mapsto  e_{j+1,j}$ and $q^{\varepsilon_j} \mapsto ^{e_{j,j}}.$
Then the twist becomes $F = I \otimes {\mathbb V}$ and the twisted coproducts are
\begin{eqnarray}
\Delta(q^{e_{i,i}}) = q^{e_{i,i}} \otimes 
q^{e_{\tau(i), \tau(i)}}, ~~~~\Delta(\xi_j ) = \xi_j \otimes   q^{{h_{\tau(j)}\over 2}} + q^{-{h_j\over 2}} \otimes  \xi_{\tau(j)} . \label{symm2q}
\end{eqnarray}
 $h_j =\big (e_{j,j} -e_{j+1, j+1} \big )$, $h_{\tau(j)} = \big (e_{\tau(j), \tau(j)} -e_{\tau(j+1), \tau(j+1)}\big )$, for 
$\xi_j \in \Big  \{ e_{j, j+1},\ e_{j+1, j}\Big \}$,  
we define respectively: $\xi_{\tau(j)} \in \Big \{ e_{\tau(j), \tau(j+1)},\  e_{\tau(j+1), \tau(j)}\Big \}$.
We also recall that ${\mathrm g}$ (\ref{braidq})  is $\mathfrak{U}_q(\mathfrak{gl}_n)$-invariant, i.e. 
$\big [ {\mathrm g},\  \Delta_q(Y) \big ] =0,$  $~Y \in\Big  \{ e_{j, j+1},\ e_{j+1,j},\ q^{e_{j,j}} \Big \} $ and  
the co-products $\Delta_q$ of the algebra elements are given in (\ref{copba}), (\ref{copbb}) and (\ref{notation2q}).
Then it follows that the element ${\mathrm G}$  (\ref{specialq}) is also ${\mathfrak U}_q(\mathfrak{gl}_n)$ symmetric \cite{DoiSmo2}, i.e.
$\big [ {\mathrm G},\  \Delta(Y) \big ] =0,$
where  the modified co-products are given in (\ref{symm2q}). Naturally the Baxterized 
$\check R$-matrix is also $\mathfrak{U}_q(\mathfrak{gl}_n)$-invariant (\ref{q21}), (\ref{symm2q}).

As in the case of the Yangian  the admissible twist is $F= I \otimes {\mathbb V}$ 
and $\Phi = I \otimes I \otimes {\mathbb V}^{-1},$ then we have:
$F_{1,.23} = (\mbox{id} \otimes \Delta)F = I \otimes {\mathbb V} \otimes {\mathbb V}$ and $F_{12,3}^* = \big( (\Delta \otimes \mbox{id})F \big )\Phi^{-1}= I \otimes I \otimes {\mathbb V}^2,$
in accordance to the Conjecture \ref{conj1}.
Also, the $N$-fold twist is $ F_{12...N} = I \otimes {\mathbb V} \otimes {\mathbb V}^{2}\otimes \ldots \otimes {\mathbb V}^{(N-1)},$  (see also \cite{Doikoutw}).

\subsection*{Acknowledgments}

\noindent We  are grateful to A. Smoktunowicz and B. Rybolowicz for useful discussions.
Support from the EPSRC research grants EP/V008129/1 and EP/R009465/1 
is acknowledged. A.G. acknowledges 
support from Heriot-Watt University via a James Watt scholarship.

\end{document}